\documentclass[11pt,reqno]{amsart}
\usepackage[margin=1.15in]{geometry}
\usepackage{amsmath,amssymb,amsthm}

\newcommand{\Hyp}{\mathbb H^3_1}
\newcommand{\sltwo}{\mathfrak{sl}(2,\mathbb R)}
\newcommand{\SLR}{SL(2,\mathbb R)}
\newcommand{\Ad}{\mathrm{Ad}}
\newcommand{\Sph}{\mathbb S^2_1}
\newcommand{\Id}{\mathrm{Id}}

\numberwithin{equation}{section}

\theoremstyle{plain}
\newtheorem{theorem}{Theorem}[section]
\newtheorem{proposition}[theorem]{Proposition}
\newtheorem{lemma}[theorem]{Lemma}
\newtheorem{corollary}[theorem]{Corollary}

\theoremstyle{definition}
\newtheorem{problem}[theorem]{Problem}
\newtheorem{example}[theorem]{Example}

\theoremstyle{remark}
\newtheorem{remark}[theorem]{Remark}

\begin{document}

\title[Timelike CMC and CGC surfaces in anti-de Sitter $3$-space]{Timelike surfaces of constant mean and constant Gaussian curvature in anti-de Sitter $3$-space via loop groups}

\author{Jorge BRAVO-GADEA}

\address{Departamento de F\'isica y Matem\'aticas, Universidad de Alcal\'a\\
Alcal\'a de Henares, Madrid, Spain}
\email{jorge.bravog@uah.es}


\begin{abstract}
The loop group description of surfaces whose Gauss map is Lorentz harmonic is known for timelike surfaces of constant mean curvature in the flat Lorentzian $3$-space, and for surfaces of constant Gaussian curvature in the $3$-sphere. We carry out the corresponding construction for timelike surfaces in anti-de Sitter $3$-space $\Hyp\cong\SLR$. These are the surfaces of constant mean curvature, whose Gauss map is harmonic for the conformal structure of the first fundamental form, and those of constant Gaussian curvature $K>-1$, $K\ne0$, whose Gauss map is then an immersion and is harmonic for the second fundamental form. For both classes we prove the harmonicity characterization and recover the surfaces from extended frames by Sym--Bobenko-type formulas. We also solve the geometric Cauchy problem by the generalized DPW method, with potentials written explicitly in terms of the data and no normalization of the Gauss map. We then show that the two constructions differ only by a normalization. After the substitution $\lambda=\zeta^2$ and a constant gauge, a constant mean curvature extended frame $\hat F$ is an extended frame of its Gauss map, and the maps
\[
\hat F|_{\lambda=e^{4\theta}}\,\exp(\theta e_3)\,\hat F|_{\lambda=1}^{-1},\qquad e_3\,\hat F|_{\lambda=e^{4\theta}}\,e_3\,\exp(\theta e_3)\,\hat F|_{\lambda=1}^{-1},
\]
with $e_3=\operatorname{diag}(-1,1)$, are timelike surfaces of constant Gaussian curvature $1/\sinh^2\theta$ and $-1/\cosh^2\theta$ wherever the constant mean curvature surface is not flat. The first is the classical parallel surface, and the second is its polar surface. Explicit families, with extended frames and surfaces in closed form, illustrate the constructions.
\end{abstract}

\maketitle

\section{Introduction}
\label{sec:intro}

A recurring theme in the theory of surfaces in three-dimensional homogeneous spaces is the correspondence between a curvature condition on an immersion and an analytic condition on its Gauss map. The classical statement in Euclidean space is the theorem of Ruh and Vilms~\cite{RV70}: an immersion has constant mean curvature if and only if its Gauss map is harmonic. Surfaces of constant negative Gaussian curvature are governed instead by the sine-Gordon equation~\cite{BP96}, and in hyperbolic space the surfaces of mean curvature one admit Bryant's holomorphic representation~\cite{Bryant87}. In each case the geometric condition is an integrable system in disguise, and the systematic way to exploit this is through loop groups: harmonic maps into a symmetric space are described by loop group factorizations~\cite{PS86,Guest97,HSW13}, and the method of Dorfmeister, Pedit and Wu~\cite{DPW98}, together with its generalizations~\cite{BD09}, turns the construction of the corresponding surfaces into the choice of a pair of potentials followed by a Birkhoff factorization. See~\cite{FKR06} for a detailed exposition in the constant mean curvature case.

A second recurring theme, and the one this paper is ultimately about, is the existence of correspondences \emph{between} surface classes: the Lawson correspondence between minimal surfaces in $\mathbb R^3$ and constant mean curvature surfaces in $S^3$~\cite{Lawson70}, and the B\"acklund and Darboux transformations of integrable surface theory~\cite{Tenenblat98,Krichever80}, which produce new surfaces from old ones and which, in the integrable picture, are typically implemented by simple operations on the loop group data. The result proved in Section~\ref{sec:bridge} below is a correspondence of that kind, and its content is that two classical geometric operations --- passing to a parallel surface, and to the polar surface of that --- become, at the level of extended frames, the insertion of a single constant element between two evaluations of the frame.

Two classes of surfaces coexist for surfaces in a $3$-dimensional space form of nonzero curvature, and they are governed by \emph{different} conformal structures on the parameter domain. If the Gauss map of a surface is harmonic with respect to the conformal structure of the \emph{first} fundamental form, the surface has constant mean curvature; if it is harmonic with respect to the \emph{second} fundamental form, the surface has constant Gaussian curvature. Both statements are classical in Euclidean space, and both have loop group counterparts. In the Lorentzian setting the picture is complicated by causality: a surface may be spacelike or timelike, the induced conformal structure on the domain may be Riemannian or Lorentzian, and the loop group formalism must be set up over a real form adapted to the case at hand.

The loop group treatment of these two surface classes has been carried out in several ambient spaces, and the present paper supplies a missing entry in that record. In the \emph{flat} Lorentzian ambient space $\mathbb R^{2,1}$, the characterization of timelike constant mean curvature surfaces by harmonicity of the Gauss map with respect to the first fundamental form goes back to T.K.~Milnor~\cite{TKM83}; the full loop group description --- twisted real form, real spectral parameter, Birkhoff decomposition, normalized potentials and a Sym formula --- is due to Dorfmeister, Inoguchi and Toda~\cite{DIT02}, with the Sym formula and the associated family of parallel surfaces already in Inoguchi~\cite{Ino98} and the minimal case in Inoguchi and Toda~\cite{IT04}. The corresponding statement for constant Gaussian curvature, namely harmonicity of the Gauss map with respect to the \emph{second} fundamental form, is stated in both~\cite[Prop.~1.9]{DIT02} and~\cite[Prop.~4.5]{IT04}. Timelike surfaces in $\mathbb R^{2,1}$ with prescribed mean curvature and Gauss map were studied, without loop groups, by Magid~\cite{Mag91}. On the same flat side, Toda~\cite{Toda05} gave the generalized d'Alembert representation for pseudospherical surfaces in $\mathbb R^3$, and Brander and Svensson solved the geometric Cauchy problem --- for the immersion itself, with a prescribed curve and tangent plane along it --- for pseudospherical surfaces in $\mathbb R^3$ and for timelike constant mean curvature surfaces in $\mathbb R^{2,1}$~\cite{BS13}, later analysing the singularities of the latter~\cite{BS14}. A unified treatment of the flat cases is due to Kobayashi~\cite{Kob11}: the real forms of the twisted $\mathfrak{sl}(2,\mathbb C)$ loop algebra fall into seven classes, six of which give the surfaces of constant mean or constant Gaussian curvature in $\mathbb R^3$ and the spacelike and timelike ones in $\mathbb R^{2,1}$, each characterized by the harmonicity or Lorentz harmonicity of its Gauss map~\cite[Thms.~3.3 and~3.6]{Kob11}. In that classification the timelike surfaces of $\mathbb R^{2,1}$ whose Gauss map is Lorentz harmonic into $\Sph$, together with their parallel surfaces of constant mean curvature, form a single class~\cite[case~$(S,2)$ and Cor.~3.4]{Kob11}; it is the flat counterpart of the pair of classes treated here.

In a \emph{curved} ambient space the program was carried out by Brander, Inoguchi and Kobayashi~\cite{BIK14} for surfaces of constant Gaussian curvature $K<1$ in the $3$-sphere, and a loop group method for constant Gaussian curvature surfaces in the Riemannian hyperbolic $3$-space has since been established by Inoguchi and Kobayashi~\cite{IK24}, where it is used to classify the weakly complete surfaces with $-1<K<0$. The paper~\cite{BIK14} contains, for $S^3$, all the structural ingredients used in Sections~\ref{sec:cgc} and~\ref{sec:bridge} below: the normal Gauss map $\nu=f^{-1}N$ obtained by left translation of the unit normal $N$ into the Lie algebra, the theorem that $\nu$ is Lorentz harmonic for the conformal structure of the second fundamental form exactly when $K$ is constant~\cite[Thm.~3.1]{BIK14}, the notion of admissible frame, in which the $\lambda$-independent part of the Maurer--Cartan form takes values in the isotropy subalgebra of the symmetric space and the coefficients of $\lambda^{\pm1}$ take values in its complement, the generalized d'Alembert representation, the recovery of the immersion from the extended frame as the \emph{algebraic} evaluation $f=\hat F|_{\lambda=\mu}\hat F|_{\lambda=1}^{-1}$ --- available because the ambient space is itself the group --- and a Lawson-type correspondence obtained by evaluating one and the same extended frame at different values of the loop parameter. For surfaces in $\Hyp$ without loop groups, timelike surfaces of constant mean curvature $\pm1$ were treated by Hong~\cite{Hong94} and, in detail, by Lee~\cite{Lee06}, whose adapted frame and Gauss--Codazzi system in $\Hyp$ are those recalled in \S\ref{sec:cmc-frame}, and whose~\cite[Thm.~4(2)]{Lee06} is the formula of Proposition~\ref{prop:sym-bobenko} for general $H$, stated there without a spectral parameter; also without loop groups, linear Weingarten surfaces of any index in a semi-Riemannian space form $M^3_s(\bar K)$ with two real distinct principal curvatures --- a class which contains timelike surfaces in $\Hyp$ --- are put in correspondence, in curvature line coordinates, with solutions of equations of sine-, sinh- and cosh-Gordon and Liouville type in~\cite[Chap.~I, \S2]{Tenenblat98}; the flat Lorentzian tori were classified by Le\'on-Guzm\'an, Mira and Pastor~\cite{LGMP11}. B\"acklund transformations between surfaces of constant Gaussian curvature in the de Sitter and anti-de Sitter $3$-spaces, timelike ones included, are constructed in Tchebyshev coordinates by Zuo, Chen and Cheng~\cite{ZCC02}, who relate them to loop group actions in the de Sitter case, and those of timelike linear Weingarten surfaces in Lorentzian space forms by Zuo~\cite{Zuo06}. The DPW method for \emph{spacelike} surfaces in Lorentzian space forms of nonzero curvature is available~\cite{Ogata17}, and the spacelike surfaces of constant Gaussian curvature in $\Hyp$, whose Gauss map takes values in the hyperbolic plane rather than in $\Sph$, are treated by loop group methods in~\cite{BG26}; and loop groups have been applied in a curved Lorentzian ambient space of a different kind by Kiyohara and Kobayashi~\cite{KK22}, who characterize timelike minimal surfaces in the Lorentzian Heisenberg group by the harmonicity of a left-translated Gauss map with values in $\Sph$ and give a generalized Weierstrass representation for them, and by Brander and Kobayashi~\cite{BK25}, for surfaces of zero mean curvature in the same ambient space. Lorentzian harmonic maps into the de Sitter $2$-space --- the target of the Gauss map throughout this paper --- and the timelike surfaces associated with them have also been studied by other means: Akamine and Kiyohara~\cite{AK26} construct such maps, and timelike constant mean curvature surfaces in $\mathbb R^{2,1}$ and timelike minimal surfaces in the Lorentzian Heisenberg group, from framed null curves rather than from potentials.

What is missing from this record is the case that is simultaneously curved and Lorentzian with a Lorentzian \emph{induced} metric: timelike surfaces in $\Hyp$. The omission is not accidental. The construction of~\cite{BIK14} takes place over the compact symmetric space $S^2=SU(2)/U(1)$, while the target here is the de Sitter $2$-sphere, whose stabilizer $SO(1,1)$ is noncompact and whose metric is indefinite; the general frameworks for constant curvature immersions into pseudo-Riemannian space forms by curved flats~\cite[Thm.~3.2]{Brander07} require, by contrast, that the \emph{induced} metric be positive definite; what is used here is the generalized DPW method of~\cite{BD09}, not that application of it. Inoguchi and Toda, having observed that timelike minimal surfaces in $\Hyp$ are harmonic maps into $\SLR$, write in~\cite[Rem.~10.3]{IT04} that they hope a general scheme would allow substantial progress in the study of such surfaces; the present paper is a contribution in that direction.

One treatment in the literature does reach the surfaces of Section~\ref{sec:cgc}. Xia~\cite{Xia07} studies surfaces of constant Gaussian curvature in pseudo-Riemannian three-dimensional space forms $\bar M^3_s(\bar K)$ by loop group methods, and the range of indices treated there includes timelike surfaces in $\Hyp$. Theorem~2.1 of~\cite{Xia07} puts such a surface into coordinates adapted to its asymptotic directions, in which the Tchebyshev angle satisfies a sinh-Gordon equation when $K>\bar K$ and the principal curvatures are real, a cosh-Gordon equation when $K>\bar K$ and they are imaginary, and a sine-Laplace equation when $K<\bar K$; these systems are then written as the flatness, for every value of a real or complex loop parameter, of a one-form with values in the loop algebra of $SO_J(4)$, $J=\operatorname{diag}(\bar K,\epsilon_1,\epsilon_2,\epsilon_3)$, twisted by conjugation with $\operatorname{diag}(-1,1,1,-1)$. For the ambient space of the present paper that algebra is $\mathfrak{so}(2,2)\cong\sltwo\oplus\sltwo$, and the description splits accordingly into two chiral halves, of which one is the formalism used in Section~\ref{sec:cgc} below: the same spectral parameter and the same $\lambda$-supports, with the left-translated normal $\nu=f^{-1}N$ of \S\ref{sec:cgc-cauchy} in one factor and the right-translated $Nf^{-1}$ in the other. A Birkhoff decomposition is proved for the resulting loop groups~\cite[Thms.~4.2--4.4]{Xia07} and normalized potentials are obtained in each case in terms of the initial values of the Tchebyshev angle along the two characteristics; the timelike surfaces in $\Hyp$ with $K>\bar K$ are treated in~\cite[Cors.~4.2 and~4.3]{Xia07}. A generalized Weierstrass representation of those surfaces is therefore already available, and none is claimed here. It does not cover every surface of Section~\ref{sec:cgc}: in the case of real principal curvatures~\cite[Thm.~2.1]{Xia07} produces coordinates in which $I$ and $II$ are simultaneously diagonal, which presupposes a diagonalizable shape operator, whereas a timelike surface whose shape operator is a nontrivial Jordan block, with a real double eigenvalue $k\ne0$, has $K+1=k^2>0$ and $\det II<0$ and falls under Section~\ref{sec:cgc}. In asymptotic coordinates such a surface has exactly one of $\langle f_x,f_x\rangle$, $\langle f_y,f_y\rangle$ equal to zero, and the surfaces of Example~\ref{ex:akamine-kiyohara} are of this kind. The loop group formalism of Section~\ref{sec:cgc} is thus, up to that splitting, the one of~\cite{Xia07}, and no priority is claimed for it. What is absent from that treatment is the ingredient the present paper is organized around: the Gauss map does not appear there, so neither the identification of one chiral half of the frame as the normal Gauss map nor the characterization of the curvature condition by its harmonicity is available; the immersion is read off the frame at a single value of the loop parameter rather than from an evaluation at two, so there is no formula of Sym--Bobenko type; the recovery of the surface carries no free parameter, whereas Proposition~\ref{prop:cgc-converse} produces a one-parameter family; and constant mean curvature surfaces, and with them any correspondence of the kind proved in Section~\ref{sec:bridge}, are not treated.

Section~\ref{sec:cmc} concerns the same surfaces as~\cite{Hong94,Lee06}, and the relation is as follows. The construction of Section~\ref{sec:cmc} produces mean curvature $H=(\mu+1)/(\mu-1)$, where $\mu$ is the spectral parameter; this never takes the values $\pm1$, which are the degenerate limits $\mu\to\infty$ and $\mu\to0$, and it satisfies $|H|>1$ for $\mu>0$ and $|H|<1$ for $\mu<0$. The values $H=\pm1$ studied in~\cite{Hong94,Lee06} are thus exactly the ones our loop group description excludes, so the two treatments are complementary; and, as we shall see in Section~\ref{sec:bridge}, it is precisely the range $|H|>1$ that admits a parallel surface of constant Gaussian curvature.

\subsection*{What this paper does}

Our aim is not only to record the two constructions, but to exhibit the relation between them. The paper is organized around the following three steps.

\emph{Step 1 (Section~\ref{sec:cmc}).} We set up the theory of timelike constant mean curvature (CMC) surfaces in $\Hyp$. Working in null coordinates for the first fundamental form and in a frame adapted to the two null tangent directions, we compute the frame coefficients in terms of the geometric data $(H,Q,R,\omega)$ (Lemma~\ref{lem:frame-coeffs}), derive the Gauss--Codazzi system (Proposition~\ref{prop:GC}), and prove the harmonicity characterization: the Gauss map is harmonic for the first fundamental form if and only if $H$ is constant, if and only if a one-parameter family of $\sltwo$-valued one-forms $\alpha^\lambda$ satisfies the Maurer--Cartan equation for every value of the spectral parameter (Theorem~\ref{thm:cmc-harmonic}). We then prove the Sym--Bobenko formula for this situation (Proposition~\ref{prop:sym-bobenko}): from a regular admissible extended frame $\hat F$ and a parameter $\mu\ne0,1$ one obtains a timelike CMC immersion
\[
f^\mu=\hat F|_{\lambda=\mu}\,\hat F|_{\lambda=1}^{-1},\qquad H=\frac{\mu+1}{\mu-1}.
\]
Two examples (\S\ref{sec:cmc-example}) run the whole machinery in closed form, one of them with nonconstant curvature.

\emph{Step 2 (Section~\ref{sec:cgc}).} We set up the theory of timelike constant Gaussian curvature (CGC) surfaces. Here the relevant conformal structure is that of the second fundamental form, the adapted coordinates are the asymptotic ones, and the characterization (Theorem~\ref{thm:cgc-harmonic}) states that, for a timelike immersion with positive extrinsic curvature and immersive Gauss map, the Gauss map is harmonic for the second fundamental form exactly when the Gaussian curvature is constant. The converse (Proposition~\ref{prop:cgc-converse}) recovers, from a harmonic immersion $\nu$ into the de Sitter $2$-sphere and for each $\rho\ne0,\pm1$, a timelike CGC surface with $K+1=\rho^2$ by integrating an explicit one-form; once the extended frame of $\nu$ is available, the same surfaces are read off it algebraically, by a Sym--Bobenko-type formula whose two branches $K>0$ and $-1<K<0$ are separated by the sign of the spectral parameter (Proposition~\ref{prop:cgc-sym}). We solve the Cauchy problem for harmonic maps into $\Sph$ and the geometric Cauchy problem for timelike CGC surfaces (Theorem~\ref{thm:cauchy} and Corollary~\ref{cor:geometric-cauchy}); Example~\ref{ex:cauchy} carries both out on data for which the potential, the solution and the resulting surfaces are all written down explicitly.

\emph{Step 3 (Section~\ref{sec:bridge}).} The two constructions of Steps~1 and~2 are, on the face of it, unrelated: they use different conformal structures, different adapted frames and different potentials, and each recovers only surfaces of its own class from its own frame. We first recall the parallel surface correspondence (Theorem~\ref{thm:parallel}), which relates the two classes at the level of individual surfaces, and then show that the two constructions are related by an explicit formula. The corresponding statement in the flat ambient space $\mathbb R^{2,1}$ is due to Inoguchi~\cite[\S4]{Ino98}, where the parallel surface at distance $1/(2H)$ from a timelike CMC $H$ surface is shown to have constant Gaussian curvature $4H^2$, and the parallel surface at distance $1/H$ to have mean curvature $-H$. For timelike surfaces in the Lorentzian space forms of nonzero curvature, $\Hyp$ included, the curvatures of the parallel surfaces, and the constant mean curvature surfaces parallel to a surface of constant Gaussian curvature, are in Zuo~\cite[Thm.~3.4 and Rem.~3.5]{Zuo06}. So the correspondence itself is known; what follows is its expression on the extended frame, where the parallel displacement is implemented inside the group. Writing $N^\mu$ for the unit normal of $f^\mu$, the parallel surface $f^\mu\cosh\theta+N^\mu\sinh\theta$ is, as an identity of matrices,
\[
f^\mu\cosh\theta+N^\mu\sinh\theta=\hat F|_{\lambda=\mu}\,\exp(\theta e_3)\,\hat F|_{\lambda=1}^{-1},
\]
and when the spectral parameter is tied to the parallel distance by $\mu=e^{4\theta}$, the resulting immersion has constant Gaussian curvature $K=1/\sinh^2\theta$ at every point where the CMC surface $f^\mu$ is not flat (Theorem~\ref{thm:bridge}). Thus a single CMC extended frame, evaluated at two values of $\lambda$ and multiplied by a fixed one-parameter subgroup element, produces CGC surfaces, with no further integration. Parallel displacement reaches only $K>0$: a CGC surface with $K\le0$ and nonconstant mean curvature has no parallel surface of constant mean curvature (Proposition~\ref{prop:no-parallel-cmc}). The formula works for a structural reason. After the substitution $\lambda=\zeta^2$ and a constant gauge, a CMC extended frame is an extended frame of its Gauss map in the sense of Step~2 (Lemma~\ref{lem:cmc-twisted}). Evaluating it at $\zeta=-e^{2\theta}$ instead of $e^{2\theta}$ gives
\[
e_3\,\hat F|_{\lambda=e^{4\theta}}\,e_3\,\exp(\theta e_3)\,\hat F|_{\lambda=1}^{-1},
\]
of constant Gaussian curvature $-1/\cosh^2\theta$, which is the polar surface of the parallel one. So both branches are reached from a single CMC frame (Theorem~\ref{thm:bridge}). The same observation solves the geometric Cauchy problem for CMC surfaces (Corollary~\ref{cor:cauchy-cmc}).

In the language of integrable systems, Steps~1--3 say the following. Both classes are governed by one integrable system, the Lorentz harmonic maps into $\Sph$ in the twisted loop group formulation of \S\ref{sec:cgc-cauchy}; they differ only in which conformal structure on the parameter domain is used, that of the first fundamental form for constant mean curvature and that of the second for constant Gaussian curvature. Lemma~\ref{lem:cmc-twisted} identifies the two formulations through $\lambda=\zeta^2$, so that the spectral parameter of the constant mean curvature system is the square of the twisted one. In these terms, the parallel surface and its polar correspond to evaluating one twisted frame at $\zeta=e^{2\theta}$ and at $\zeta=-e^{2\theta}$, respectively, against its value at $\zeta=1$. The geometric Cauchy problems are, on both sides, non-characteristic Cauchy problems for that system, solved by the generalized d'Alembert method of~\cite{BD09}. The existence part uses nothing beyond~\cite{BD09}; what is new is the explicit form of the potentials in terms of the geometric data. That one system, in one real form, governs a constant mean curvature class and a constant Gaussian curvature class at once is, for the flat ambient spaces, the point of view of~\cite{Kob11}.

The same formula with $\exp(2\theta e_3)$ in place of $\exp(\theta e_3)$ returns to the constant mean curvature class, with the mean curvature reversed (Corollary~\ref{cor:bridge-cmc}); in fact the whole one-parameter family obtained by letting the inserted group element run has a closed-form description (Lemma~\ref{lem:interpolation}), every member of it is a linear Weingarten surface (Proposition~\ref{prop:lw}), and three of its members are distinguished --- two of constant mean curvature at the ends and one of constant Gaussian curvature at the midpoint, with three genuinely different nondegeneracy conditions (Remark~\ref{rem:three-values}). The same three distinguished members are present in the flat case~\cite[Props.~4.1 and~4.3]{Ino98}, and the pattern is the one the logarithmic term of the Sym formula produces there~\cite[\S5]{DIT02}; what changes in $\Hyp$ is that no differentiation in $\lambda$ is involved, the whole family being obtained by inserting a one-parameter subgroup between two evaluations of the same extended frame.

The hypothesis $K\ne0$ needed at the midpoint is exactly the standing assumption of Section~\ref{sec:cgc}, that the Gauss map be an immersion (Proposition~\ref{prop:degenerate-locus}). Since an extended frame with constant coefficients always violates it (Proposition~\ref{prop:constant-potential-flat}), Example~\ref{ex:sinh-gordon} and Remark~\ref{rem:bridge-example} give an explicit family for which it holds at every point.

Much of Sections~\ref{sec:cmc} and~\ref{sec:cgc} transposes to timelike surfaces in $\Hyp$ the constructions of~\cite{DIT02,Lee06} and of~\cite{BIK14,Xia07}; the corresponding results carry the attribution in their headers, and the arguments are given in full because the target $\Sph$ has noncompact stabilizer and the induced metric is indefinite. The results for which we know no precedent are: the characterization of constant Gaussian curvature by harmonicity of the normal Gauss map for timelike surfaces in $\Hyp$ (Theorem~\ref{thm:cgc-harmonic}), with its converse in both branches (Proposition~\ref{prop:cgc-converse}); the characterization of the degenerate locus (Proposition~\ref{prop:degenerate-locus}); the absence of parallel surfaces of constant mean curvature for constant Gaussian curvature surfaces with $K\le0$ and nonconstant mean curvature (Proposition~\ref{prop:no-parallel-cmc}); the solution of the Cauchy problem for harmonic maps into $\Sph$ from data on a non-characteristic curve, and of the geometric Cauchy problems for timelike constant Gaussian and constant mean curvature surfaces (Theorem~\ref{thm:cauchy} and Corollaries~\ref{cor:geometric-cauchy} and~\ref{cor:cauchy-cmc}, with Remark~\ref{rem:H-small} for $|H|<1$); the formulas of Theorem~\ref{thm:bridge}, which give both curvature branches from a single constant mean curvature frame, together with Lemma~\ref{lem:interpolation} and Corollary~\ref{cor:bridge-cmc}; and the explicit families of Examples~\ref{ex:sinh-gordon}, \ref{ex:cauchy} and~\ref{ex:akamine-kiyohara}, in which the surfaces are given in closed form.

The relation to the author's preprint~\cite{BG26}, which treats \emph{spacelike} surfaces in the same ambient space, is as follows. The strategy of Section~\ref{sec:cgc} is the one followed there. Three things are different in the timelike case. The Gauss map takes values in $\Sph$, whose stabilizer is noncompact. A null derivative of $\nu$ may be spacelike, timelike or null, so the frame cannot be normalized by the length of $\nu_x$ as in~\cite[\S2.3]{BG26}, and the loop group description of \S\ref{sec:cgc-cauchy} is set up without any normalization (Example~\ref{ex:akamine-kiyohara}). And the first fundamental form is itself Lorentzian, so the constant mean curvature theory of Section~\ref{sec:cmc}, and with it the whole of Section~\ref{sec:bridge}, has no counterpart in~\cite{BG26}.

The singularities of the immersions constructed here, and the behaviour of the constructions at the points where the various nondegeneracy conditions fail, are deliberately left outside the scope of this paper. The singularities are governed by the failure of the Birkhoff decomposition away from the big cell, which in the noncompact setting used here is not global; this is the phenomenon isolated in~\cite[\S1]{BS13} and analysed, for the flat ambient space, in~\cite{BS14}.

\section{Preliminaries}
\label{sec:prelim}

This section fixes notation and recalls, without proof, the standard material used in both constructions: the Lie group model of the ambient space, the second fundamental form and the Gauss map, harmonic maps into the de Sitter $2$-sphere, and the loop group formalism. Nothing here is new.

\subsection{Anti-de Sitter $3$-space as a Lie group}
\label{sec:prelim-liegroup}

Let $\mathfrak{gl}(2,\mathbb R)$ be the space of real $2\times2$ matrices, with basis
\[
e_0=\begin{pmatrix}1&0\\0&1\end{pmatrix},\quad e_1=\begin{pmatrix}0&-1\\1&0\end{pmatrix},\quad e_2=\begin{pmatrix}0&1\\1&0\end{pmatrix},\quad e_3=\begin{pmatrix}-1&0\\0&1\end{pmatrix}.
\]
Equip $\mathfrak{gl}(2,\mathbb R)$ with the bilinear form $\langle X,Y\rangle=-\tfrac12\operatorname{tr}(X\bar Y)$, where $\bar Y$ denotes the adjoint (cofactor) transpose of $Y$, so that $Y\bar Y=(\det Y)\Id$, where $\Id$ denotes the identity matrix. This makes $\mathfrak{gl}(2,\mathbb R)$ isometric to $\mathbb R^{2,2}$, with $\langle X,X\rangle=-\det X$, and therefore
\[
\Hyp=\{X\in\mathfrak{gl}(2,\mathbb R)\mid\langle X,X\rangle=-1\}=\{X\mid\det X=1\}=\SLR
\]
realizes anti-de Sitter $3$-space as a quadric in $\mathbb R^{2,2}$~\cite{ONeill83,Milnor76,Lee06b}, with a bi-invariant Lorentzian metric of constant sectional curvature $-1$. Explicitly, the isometry $\mathbb R^{2,2}\to\mathfrak{gl}(2,\mathbb R)$ may be taken to be
\begin{equation}
\label{eq:identification}
(p_1,p_2,p_3,p_4)\longmapsto\begin{pmatrix}p_1-p_4&p_3-p_2\\p_3+p_2&p_1+p_4\end{pmatrix},
\end{equation}
under which $\langle X,X\rangle=-\det X=-p_1^2-p_2^2+p_3^2+p_4^2$; we use it in \S\ref{sec:cgc} for the example of a flat surface. Restricting to the Lie algebra $\sltwo=\operatorname{span}\{e_1,e_2,e_3\}$ gives the inner product $\langle X,Y\rangle=\tfrac12\operatorname{tr}(XY)$ on $\sltwo$, with
\[
\langle e_1,e_1\rangle=-1,\qquad\langle e_2,e_2\rangle=\langle e_3,e_3\rangle=1,
\]
so that $\sltwo\cong\mathbb R^{2,1}$, and
\begin{equation}
\label{eq:e-brackets}
[e_1,e_2]=2e_3,\qquad [e_2,e_3]=-2e_1,\qquad [e_3,e_1]=2e_2,
\end{equation}
that is, the Lie bracket is twice the Lorentzian cross product, $[X,Y]=2(X\times Y)$. The causal character of an element of $\sltwo$ --- spacelike, timelike or null --- coincides with its type as a matrix (semisimple with real, respectively imaginary, eigenvalues, or nilpotent), and is therefore invariant under conjugation. We write
\[
\Sph=\{X\in\sltwo\mid\langle X,X\rangle=1\}
\]
for the de Sitter $2$-sphere, a Lorentzian surface of constant curvature $1$; it is the target of the Gauss map of a timelike immersion.

The following identities are elementary consequences of the above and are used constantly.

\begin{proposition}
\label{prop:bracket-identities}
Let $X,Y,Z\in\sltwo$. Then
\begin{align*}
&\text{\rm(i) } \langle X,[X,Y]\rangle=0, &&\text{\rm(ii) } \langle[X,Y],[X,Y]\rangle=4\big(\langle X,Y\rangle^2-\langle X,X\rangle\langle Y,Y\rangle\big),\\
&\text{\rm(iii) } [[Z,X],[Z,Y]]=-4\langle Z,Z\rangle[X,Y], &&\text{\rm(iv) } \langle[Z,X],[Z,Y]\rangle=-4\langle Z,Z\rangle\langle X,Y\rangle,
\end{align*}
the last two under the assumption that $Z$ is orthogonal to both $X$ and $Y$. If moreover $\langle Z,Z\rangle=1$, then
\[
\text{\rm(v) } [Z,[Z,W]]=4W\qquad\text{for every }W\perp Z .
\]
\end{proposition}

Identity (v) is checked by conjugating $Z$ to $e_3$, which is possible since every unit spacelike vector is conjugate to $e_3$, and evaluating on $e_1,e_2$ using~\eqref{eq:e-brackets}. It says that $\operatorname{ad}_Z$ restricted to the Lorentzian plane $Z^\perp$ squares to four times the identity; in particular $\operatorname{ad}_Z$ is an isomorphism of $Z^\perp$, so it carries any basis of $Z^\perp$ to a basis. We use this repeatedly in Section~\ref{sec:cgc}.

\subsection{The second fundamental form and the Gauss map}
\label{sec:prelim-gauss}

For a bi-invariant metric on a matrix Lie group the Levi-Civita connection acts on left-invariant vector fields by $\nabla_XY=\tfrac12[X,Y]$. Let $f:\Omega\to\Hyp$ be a timelike immersion of a domain $\Omega\subset\mathbb R^2$, with unit normal $N$, so that $\langle N,N\rangle=1$. The \emph{Gauss map} of $f$ is obtained by left translation to the identity,
\[
\nu:=f^{-1}N:\Omega\longrightarrow\Sph\subset\sltwo .
\]
The interplay between this Gauss map and the second fundamental form, for surfaces in a Lie group with a bi-invariant metric, is studied in general in~\cite{FP16}; the moving-frame formalism used throughout is the classical one, for which see~\cite{JMN16}. Writing $u=f^{-1}f_x$ and $v=f^{-1}f_y$ for the left-translated tangent fields, the coefficients of the first and second fundamental forms are
\begin{gather*}
E=\langle u,u\rangle,\qquad \mathcal F=\langle u,v\rangle,\qquad G=\langle v,v\rangle,\\
L=\big\langle u_x,\nu\big\rangle,\qquad M=\big\langle v_x+\tfrac12[u,v],\nu\big\rangle,\qquad \mathcal N=\big\langle v_y,\nu\big\rangle,
\end{gather*}
the second line being the Levi-Civita connection written in left-translated form. It is convenient to record that these agree with the naive second derivatives: for $X,Y\in\sltwo$ the Cayley--Hamilton identity gives $XY+YX=2\langle X,Y\rangle\Id$, a multiple of the identity matrix and hence orthogonal to $\sltwo$, so from $f^{-1}f_{xx}=u^2+u_x$ and $f^{-1}f_{xy}=v_x+uv=v_x+\tfrac12(uv+vu)+\tfrac12[u,v]$ one gets
\begin{equation}
\label{eq:naive-second-derivatives}
L=\langle f_{xx},N\rangle,\qquad M=\langle f_{xy},N\rangle,\qquad \mathcal N=\langle f_{yy},N\rangle .
\end{equation}
We write $S$ for the shape operator, characterized by
\[
S\begin{pmatrix}f_x\\f_y\end{pmatrix}=\begin{pmatrix}N_x\\N_y\end{pmatrix},
\]
and $K$, $H$ for the Gaussian and mean curvature of $f$. Since the ambient space has constant sectional curvature $-1$, the Gauss equation reads
\begin{equation}
\label{eq:gauss-general}
K+1=\det S=\frac{\det II}{\det I}=\frac{L\mathcal N-M^2}{EG-\mathcal F^2},\qquad H=-\tfrac12\operatorname{tr}S,
\end{equation}
so that $K+1$ is the extrinsic curvature. A timelike immersion always admits, locally, two families of null curves for $I$; when $K+1>0$ the second fundamental form is a Lorentzian metric as well, and the corresponding null curves are the asymptotic curves of $f$.

Finally, a smooth map $\nu:\Omega\to\Sph$ from a Lorentz surface~\cite{Weinstein96} is \emph{harmonic} if it is a critical point of the energy of the corresponding conformal structure; in null coordinates $(x,y)$ for that structure, harmonicity is equivalent to
\begin{equation}
\label{eq:harmonic}
[\nu,\nu_{xy}]=0 .
\end{equation}
Since \eqref{eq:harmonic} only involves the null directions, harmonicity depends on the conformal class of the Lorentz structure and not on the metric itself. The two constructions of this paper differ precisely in which Lorentz structure on $\Omega$ is used: the one induced by $I$ in Section~\ref{sec:cmc}, and the one induced by $II$ in Section~\ref{sec:cgc}.

\subsection{Loop groups, Birkhoff decomposition and the DPW recipe}
\label{sec:prelim-dpw}

Let $\Lambda SL(2,\mathbb C)$ denote the group of smooth loops in the spectral parameter $\lambda$, that is, of smooth maps from the unit circle to $SL(2,\mathbb C)$, and consider the reality involution $\varrho(\gamma(\lambda))=\overline{\gamma(\bar\lambda)}$, whose fixed set is the real form $\Lambda\SLR$. The two classes of frames used below sit differently with respect to twisting. The constant Gaussian curvature frames of \S\ref{sec:cgc-cauchy} satisfy the twisting condition for the involution $\sigma(\gamma)(\lambda)=\Ad_{e_3}\gamma(-\lambda)$, their $\lambda$-even part being a multiple of $e_3$ and their $\lambda^{\pm1}$ parts lying in the complement $\mathrm{span}(e_1,e_2)$, and on that side we work in the subgroup $G_\sigma$ of $\sigma$-fixed loops; this is what makes the projection $\Ad_Fe_3$ of an extended frame harmonic, as for any symmetric space. The constant mean curvature frames of \S\ref{sec:cmc-loop}, in the spectral parameter $\lambda$ used there, satisfy no such condition: by~\eqref{eq:admissible} their $\lambda$-even part contains the off-diagonal terms $A^2\xi_2\,dx$ and $B^1\xi_1\,dy$, while their $\lambda^{\pm1}$ parts are $A^1\xi_1\,dx$ and $B^2\xi_2\,dy$. They become $\sigma$-twisted after the substitution $\lambda=\zeta^2$ and a gauge by a constant diagonal matrix (Lemma~\ref{lem:cmc-twisted}). We keep the untwisted form in Section~\ref{sec:cmc}, because the Sym--Bobenko formula is simplest in it; harmonicity there comes from Theorem~\ref{thm:cmc-harmonic}. Write $\Lambda^+$, respectively $\Lambda^-$, for the subgroups of loops that extend holomorphically to the unit disc $\mathbb D$, respectively to $\mathbb P^1\setminus\mathbb D$ with diagonal value at $\infty$, and add a subscript $*$ for those normalized to the identity at $\lambda=0$, respectively $\lambda=\infty$. On the fixed-point set of $\varrho$ the spectral parameter is real, and it is in that real form that the timelike theory takes place; accordingly the potentials in Proposition~\ref{prop:dpw} are not required to be holomorphic in $\lambda$, and the recipe used below is the generalized one of~\cite{BD09}. The extended frames used below have Maurer--Cartan forms which are Laurent polynomials in $\lambda$; with an initial condition independent of $\lambda$ they are therefore holomorphic in $\lambda\in\mathbb C^*$, and they are evaluated below at real values of $\lambda$ which need not lie on the unit circle.

\begin{theorem}[Birkhoff decomposition, {\cite[Ch.~8]{PS86}}; real forms after~\cite{BD09}]
\label{thm:birkhoff}
Every loop $\gamma$ in the loop group $\Lambda GL(2,\mathbb R)$, defined as $\Lambda\SLR$ with $GL$ in place of $SL$, factors as $\gamma=\gamma_-D\gamma_+$ with $\gamma_\pm\in\Lambda^\pm$ and $D=\operatorname{diag}(\lambda^{k_1},\lambda^{k_2})$, and the set where $D$ can be taken to be $\Id$ --- the big cell --- is open and dense. For the real form $G:=\Lambda SL(2,\mathbb C)^{\varrho}=\Lambda\SLR$, and likewise for its subgroup $G_\sigma$ of $\sigma$-fixed loops, every element of the big cell $G_*^-\cdot G^+$ factors uniquely as $\hat F=\hat F_-\hat H_+$ with $\hat F_-\in G_*^-$ and $\hat H_+\in G^+$, and every element of $G_*^+\cdot G^-$ as $\hat F=\hat F_+\hat H_-$ with $\hat F_+\in G_*^+$ and $\hat H_-\in G^-$; here $G^\pm:=G\cap\Lambda^\pm$ and $G_*^\pm:=G\cap\Lambda_*^\pm$, and similarly for $G_\sigma$.
\end{theorem}

\begin{proposition}[the DPW recipe~\cite{DPW98}, in the generalized form of~\cite{BD09}]
\label{prop:dpw}
Given a pair of potentials
\[
\mathcal X(x)=\sum_{k\le1}\mathcal A_k\lambda^{k}\,dx,\qquad \mathcal Y(y)=\sum_{k\ge-1}\mathcal B_k\lambda^{k}\,dy,
\]
integrate $X^{-1}dX=\mathcal X$ and $Y^{-1}dY=\mathcal Y$ with prescribed initial conditions, perform the Birkhoff factorization $X^{-1}(x)Y(y)=H_-(x,y)H_+(x,y)$, and set
\[
\hat F(x,y)=X(x)H_-(x,y)=Y(y)H_+^{-1}(x,y).
\]
This is available on the open set of those $(x,y)$ for which $X(x)^{-1}Y(y)$ lies in the big cell of Theorem~\ref{thm:birkhoff}, and not in general on the whole of the domain: the big cell is open and dense but, the group here being noncompact, it is not everything. The set does contain the locus $X(x)=Y(y)$, where the factorization is trivial. On that open set, $\hat F^{-1}d\hat F$ has the \emph{admissible} form $(C_0+C_1\lambda)\,dx+(\mathcal D_{-1}\lambda^{-1}+\mathcal D_0)\,dy$ for some $\sltwo$-valued functions $C_i,\mathcal D_i$ of $(x,y)$ --- that is, its $\lambda$-support is $\{0,1\}$ in $dx$ and $\{-1,0\}$ in $dy$. Conversely every extended frame whose Maurer--Cartan form has this shape arises in this way, from potentials obtained by factorizing $\hat F$ itself.
\end{proposition}

Proposition~\ref{prop:dpw} is what converts the harmonic map problem into the choice of two potentials together with a factorization. Sections~\ref{sec:cmc} and~\ref{sec:cgc} specialize it, in two different ways.

\section{Timelike surfaces of constant mean curvature}
\label{sec:cmc}

In this section the Gauss map is treated as a map from the Lorentz surface $(\Omega,I)$, that is, harmonicity is measured against the conformal structure of the \emph{first} fundamental form. Throughout, $f:\Omega\to\Hyp$ is a timelike immersion with unit normal $N$, and $(x,y)$ are null coordinates for $I$, so that $\langle f_x,f_x\rangle=\langle f_y,f_y\rangle=0$ and $I=\varepsilon e^\omega\,dx\,dy$ in the sense that $2\langle f_x,f_y\rangle=\varepsilon e^\omega$ with $\varepsilon=\pm1$.

\subsection{The adapted frame and the Gauss--Codazzi system}
\label{sec:cmc-frame}

It is convenient to replace the basis $\{e_1,e_2,e_3\}$ of $\sltwo$ by a basis adapted to the null directions,
\[
\xi_1=\frac{e_1+e_2}2=\begin{pmatrix}0&0\\1&0\end{pmatrix},\qquad \xi_2=\frac{-e_1+e_2}2=\begin{pmatrix}0&1\\0&0\end{pmatrix},
\]
for which $\langle\xi_1,\xi_1\rangle=\langle\xi_2,\xi_2\rangle=0$, $\langle\xi_1,\xi_2\rangle=\tfrac12$ and $\langle\xi_i,e_3\rangle=0$. From \eqref{eq:e-brackets},
\begin{equation}
\label{eq:xi-brackets}
[\xi_1,\xi_2]=e_3,\qquad [\xi_1,e_3]=-2\xi_1,\qquad [\xi_2,e_3]=2\xi_2,
\end{equation}
while as matrices
\begin{equation}
\label{eq:xi-products}
\xi_1^2=\xi_2^2=0,\qquad \xi_1\xi_2=\tfrac12(\Id+e_3),\qquad \xi_2\xi_1=\tfrac12(\Id-e_3).
\end{equation}
The nilpotency in \eqref{eq:xi-products} is what makes the computations below short, and it has no counterpart in the spacelike theory.

We shall use a frame in which both null tangent directions are scaled by the \emph{same} factor. That such a frame exists is not a property of the coordinates and cannot be arranged by reparametrizing them; it is a property of the gauge freedom in the choice of frame, as the following lemma shows.

\begin{lemma}[the adapted frame]
\label{lem:adapted-frame}
Let $f$ be as above. Then, after interchanging the null coordinates $x$ and $y$ if necessary, there exist $\varepsilon_1,\varepsilon_2\in\{\pm1\}$ with $\varepsilon_1\varepsilon_2=\varepsilon$, locally constant, and a frame $F:\Omega\to\SLR$ such that
\begin{equation}
\label{eq:adapted-frame}
f^{-1}f_x=\varepsilon_1e^{\omega/2}\Ad_F\xi_1,\qquad f^{-1}f_y=\varepsilon_2e^{\omega/2}\Ad_F\xi_2,\qquad N=f\Ad_Fe_3 .
\end{equation}
Such an $F$ is unique up to sign, so that $\Ad_F$ is uniquely determined.
\end{lemma}
\begin{proof}
Since $\nu$ is unit spacelike, $\nu^\perp$ is a Lorentzian plane and meets the null cone of $\sltwo$ in exactly two lines, which carry $f^{-1}f_x$ and $f^{-1}f_y$ respectively. Pick $F_0$ with $\nu=\Ad_{F_0}e_3$; then $\Ad_{F_0}$ carries the two null lines of $e_3^\perp$, namely $\mathbb R\xi_1$ and $\mathbb R\xi_2$, to those two lines. Interchanging the null coordinates $x$ and $y$ if necessary --- which is harmless, as they carry no preferred order --- we may assume that $\mathbb R\Ad_{F_0}\xi_1$ is the line of $f^{-1}f_x$, so that
\[
f^{-1}f_x=\kappa_1\Ad_{F_0}\xi_1,\qquad f^{-1}f_y=\kappa_2\Ad_{F_0}\xi_2
\]
for nowhere-vanishing functions $\kappa_1,\kappa_2$ with $\kappa_1\kappa_2=2\langle f_x,f_y\rangle=\varepsilon e^{\omega}$. In general $|\kappa_1|\ne|\kappa_2|$, and no reparametrization $x\mapsto\varphi(x)$, $y\mapsto\psi(y)$ preserving the null coordinates can be expected to correct this, since it would require $|\kappa_1/\kappa_2|$ to factor as $\varphi'(x)/\psi'(y)$. The correction comes instead from the residual gauge freedom $F_0\mapsto F_0\exp(\tau e_3)$, which preserves $\nu$. Since $[e_3,\xi_1]=2\xi_1$ and $[e_3,\xi_2]=-2\xi_2$ by~\eqref{eq:xi-brackets},
\[
\Ad_{\exp(\tau e_3)}\xi_1=e^{2\tau}\xi_1,\qquad \Ad_{\exp(\tau e_3)}\xi_2=e^{-2\tau}\xi_2,
\]
so this gauge rescales the two null directions by reciprocal factors and leaves the product $\kappa_1\kappa_2$, hence $\varepsilon e^\omega$, unchanged. Setting
\[
\tau=\tfrac14\log\left|\frac{\kappa_1}{\kappa_2}\right|,\qquad F=F_0\exp(\tau e_3),\qquad \varepsilon_i=\operatorname{sign}\kappa_i,
\]
one gets $\kappa_1=\varepsilon_1e^{\omega/2}e^{2\tau}$ and $\kappa_2=\varepsilon_2e^{\omega/2}e^{-2\tau}$ with $e^{\omega}=|\kappa_1\kappa_2|$, which is~\eqref{eq:adapted-frame}. The signs $\varepsilon_i$ are locally constant because $\kappa_i$ is continuous and nowhere zero, and $\varepsilon_1\varepsilon_2=\operatorname{sign}(\kappa_1\kappa_2)=\varepsilon$. Finally, the stabilizer of $e_3$ in $\SLR$ is $\{\pm\exp(\tau'e_3)\}$, and a gauge $\exp(\tau'e_3)$ preserves~\eqref{eq:adapted-frame} only if $e^{2\tau'}=e^{-2\tau'}=1$, that is $\tau'=0$; the residual freedom is therefore $F\mapsto-F$, which does not change $\Ad_F$.
\end{proof}

With this frame fixed, write the Hopf differentials $Q\,dx^2$ and $R\,dy^2$, with $Q=\langle f_{xx},N\rangle$ and $R=\langle f_{yy},N\rangle$, and let the mean curvature be $H=2\varepsilon e^{-\omega}\langle f_{xy},N\rangle$, so that
\[
II=Q\,dx^2+\varepsilon e^\omega H\,dx\,dy+R\,dy^2 .
\]

\begin{lemma}[{essentially~\cite[\S5]{Lee06}}]
\label{lem:frame-coeffs}
Writing $F^{-1}F_x=A^1\xi_1+A^2\xi_2+A^3e_3$ and $F^{-1}F_y=B^1\xi_1+B^2\xi_2+B^3e_3$, one has
\[
A^1=\tfrac12\varepsilon_1e^{\omega/2}(H-1),\qquad A^2=-\varepsilon_1e^{-\omega/2}Q,\qquad A^3=\tfrac14\omega_x,
\]
\[
B^1=\varepsilon_2e^{-\omega/2}R,\qquad B^2=-\tfrac12\varepsilon_2e^{\omega/2}(H+1),\qquad B^3=-\tfrac14\omega_y .
\]
\end{lemma}
\begin{proof}
Write $u=f^{-1}f_x=\varepsilon_1e^{\omega/2}\Ad_F\xi_1$ and $v=f^{-1}f_y=\varepsilon_2e^{\omega/2}\Ad_F\xi_2$. For any smooth map into a matrix group, $f^{-1}f_{xx}=u^2+u_x$, $f^{-1}f_{yy}=v^2+v_y$ and
\begin{equation}
\label{eq:two-expansions}
f^{-1}f_{xy}=v_x+uv=u_y+vu ,
\end{equation}
and here $u^2=v^2=0$ by~\eqref{eq:xi-products}. Differentiating and using~\eqref{eq:xi-brackets},
\begin{align*}
u_x&=\varepsilon_1e^{\omega/2}\Ad_F\big((\tfrac{\omega_x}2+2A^3)\xi_1-A^2e_3\big), & v_y&=\varepsilon_2e^{\omega/2}\Ad_F\big((\tfrac{\omega_y}2-2B^3)\xi_2+B^1e_3\big),\\
v_x&=\varepsilon_2e^{\omega/2}\Ad_F\big((\tfrac{\omega_x}2-2A^3)\xi_2+A^1e_3\big), & u_y&=\varepsilon_1e^{\omega/2}\Ad_F\big((\tfrac{\omega_y}2+2B^3)\xi_1-B^2e_3\big),
\end{align*}
while $uv=\tfrac\varepsilon2e^\omega(\Id+\Ad_Fe_3)$ and $vu=\tfrac\varepsilon2e^\omega(\Id-\Ad_Fe_3)$ by~\eqref{eq:xi-products}. Since $N=f\Ad_Fe_3$, the metric is bi-invariant, $\Id$ is orthogonal to $\sltwo$ and $\langle\xi_i,e_3\rangle=0$, pairing a second derivative with $N$ reads off its $e_3$-coefficient. Thus $Q=-\varepsilon_1e^{\omega/2}A^2$ and $R=\varepsilon_2e^{\omega/2}B^1$, and the two expansions~\eqref{eq:two-expansions} give
\[
\tfrac12\varepsilon e^\omega H=\langle f_{xy},N\rangle=\varepsilon_2e^{\omega/2}A^1+\tfrac\varepsilon2e^\omega=-\varepsilon_1e^{\omega/2}B^2-\tfrac\varepsilon2e^\omega ,
\]
which, since $\varepsilon=\varepsilon_1\varepsilon_2$, are the formulas for $A^1$ and $B^2$. Finally, the two expressions in~\eqref{eq:two-expansions} are the same element: $v_x+uv$ has no $\xi_1$-term and $u_y+vu$ has no $\xi_2$-term, so $\tfrac{\omega_y}2+2B^3=0$ and $\tfrac{\omega_x}2-2A^3=0$.
\end{proof}

The one-form $\alpha=F^{-1}F_x\,dx+F^{-1}F_y\,dy$ is the Maurer--Cartan form of $F$, so its integrability condition $d\alpha+\tfrac12[\alpha,\alpha]=0$ holds trivially. Once the coefficients are replaced by their expressions in $(H,Q,R,\omega)$ from Lemma~\ref{lem:frame-coeffs}, however, that condition becomes a system of partial differential equations for the geometric data alone: the Gauss--Codazzi equations of the immersion. The $\xi_1$- and $\xi_2$-components give the Codazzi equations, which involve only first derivatives and express the compatibility of $Q,R$ with $H$; the $e_3$-component gives the Gauss equation.

\begin{proposition}[Gauss--Codazzi; {essentially~\cite[\S5]{Lee06}}]
\label{prop:GC}
The functions $H,Q,R,\omega$ satisfy
\[
H_x=2\varepsilon e^{-\omega}Q_y,\qquad H_y=2\varepsilon e^{-\omega}R_x,\qquad
\varepsilon\,\omega_{xy}+\tfrac12e^{\omega}(H^2-1)-2e^{-\omega}QR=0 .
\]
Consequently the Gaussian curvature of $f$ is
\begin{equation}
\label{eq:K-cmc}
K=-2\varepsilon e^{-\omega}\omega_{xy}=H^2-1-4e^{-2\omega}QR .
\end{equation}
\end{proposition}
\begin{proof}
The Maurer--Cartan equation for $\alpha$ reads $(F^{-1}F_x)_y-(F^{-1}F_y)_x=[F^{-1}F_x,F^{-1}F_y]$. Expanding the right-hand side with \eqref{eq:xi-brackets},
\[
[F^{-1}F_x,F^{-1}F_y]=(A^1B^2-A^2B^1)e_3+2(A^3B^1-A^1B^3)\xi_1+2(A^2B^3-A^3B^2)\xi_2 ,
\]
so the system splits into
\[
A^1_y-B^1_x=2(A^3B^1-A^1B^3),\qquad A^2_y-B^2_x=2(A^2B^3-A^3B^2),\qquad A^3_y-B^3_x=A^1B^2-A^2B^1 .
\]
Substituting Lemma~\ref{lem:frame-coeffs} into the first equation, the terms proportional to $\omega_y(H-1)$ and to $\omega_xR$ cancel on both sides, leaving $\tfrac12\varepsilon_1e^{\omega/2}H_y=\varepsilon_2e^{-\omega/2}R_x$, that is $H_y=2\varepsilon e^{-\omega}R_x$; the second equation gives $H_x=2\varepsilon e^{-\omega}Q_y$ in the same way. For the third, $A^3_y-B^3_x=\tfrac12\omega_{xy}$ while
\[
A^1B^2-A^2B^1=-\tfrac14\varepsilon e^{\omega}(H^2-1)+\varepsilon e^{-\omega}QR ,
\]
so that $\tfrac12\omega_{xy}=-\tfrac14\varepsilon e^\omega(H^2-1)+\varepsilon e^{-\omega}QR$, which is the stated equation after multiplication by $2\varepsilon$. Finally, $I=\varepsilon e^\omega dx\,dy$ and $II=Q\,dx^2+\varepsilon e^\omega H\,dx\,dy+R\,dy^2$, so by~\eqref{eq:gauss-general}
\[
K+1=\frac{\det II}{\det I}=\frac{QR-\tfrac14e^{2\omega}H^2}{-\tfrac14e^{2\omega}}=H^2-4e^{-2\omega}QR ,
\]
and the Gauss equation just proved turns $H^2-1-4e^{-2\omega}QR$ into $-2\varepsilon e^{-\omega}\omega_{xy}$, which is~\eqref{eq:K-cmc}.
\end{proof}

The quantity $\omega_{xy}$ appearing in~\eqref{eq:K-cmc} controls more than the curvature. The following observation is elementary but it is what ties the two halves of this paper together, and we shall use it in Section~\ref{sec:bridge}.

\begin{proposition}
\label{prop:degenerate-locus}
Let $f:\Omega\to\Hyp$ be a timelike immersion with adapted frame $F$ and Gauss map $\nu=\Ad_Fe_3$. At every point of $\Omega$ the following are equivalent:
\begin{enumerate}\itemsep=0pt
\item[\rm(i)] the Gaussian curvature satisfies $K\ne0$;
\item[\rm(ii)] $\omega_{xy}\ne0$;
\item[\rm(iii)] $\nu$ is an immersion, that is $\nu_x$ and $\nu_y$ are linearly independent.
\end{enumerate}
\end{proposition}
\begin{proof}
The equivalence of (i) and (ii) is~\eqref{eq:K-cmc}. For (iii), differentiate $\nu=\Ad_Fe_3$ and use~\eqref{eq:xi-brackets}:
\[
\nu_x=\Ad_F[F^{-1}F_x,e_3]=\Ad_F\big(-2A^1\xi_1+2A^2\xi_2\big),\qquad
\nu_y=\Ad_F\big(-2B^1\xi_1+2B^2\xi_2\big),
\]
both of which lie in the plane spanned by $\Ad_F\xi_1,\Ad_F\xi_2$. Since $\Ad_F$ is an isomorphism, they are linearly independent if and only if
\[
\det\begin{pmatrix}-2A^1&2A^2\\-2B^1&2B^2\end{pmatrix}=-4\big(A^1B^2-A^2B^1\big)\ne0 .
\]
By the $e_3$-component of the Maurer--Cartan equation computed in the proof of Proposition~\ref{prop:GC},
\[
A^1B^2-A^2B^1=A^3_y-B^3_x=\tfrac12\omega_{xy},
\]
so (iii) is equivalent to (ii).
\end{proof}

Proposition~\ref{prop:degenerate-locus} says that the flat points of a timelike immersion are exactly the points at which its Gauss map degenerates. This is worth keeping in mind because the hypothesis that $\nu$ be an immersion is the standing assumption of the whole of Section~\ref{sec:cgc}, where it is shown to be a genuine restriction; here it reappears as the condition $K\ne0$, in a setting where no such hypothesis was imposed.

\subsection{The harmonicity characterization and the loop group formulation}
\label{sec:cmc-loop}

\begin{theorem}
\label{thm:cmc-harmonic}
Let $f:\Omega\to\Hyp$ be a timelike immersion with adapted frame $F$ and Gauss map $\nu=\Ad_Fe_3$, and let
\[
U^\lambda=A^1\lambda\,\xi_1+A^2\xi_2+A^3e_3,\qquad V^\lambda=B^1\xi_1+B^2\lambda^{-1}\xi_2+B^3e_3,\qquad \alpha^\lambda=U^\lambda\,dx+V^\lambda\,dy .
\]
The following are equivalent:
\begin{enumerate}\itemsep=0pt
\item[\rm(1)] $\nu$ is harmonic with respect to the conformal structure of the first fundamental form, that is $[\nu,\nu_{xy}]=0$;
\item[\rm(2)] the mean curvature $H$ is constant;
\item[\rm(3)] $\alpha^\lambda$ satisfies the Maurer--Cartan equation $d\alpha^\lambda+\tfrac12[\alpha^\lambda,\alpha^\lambda]=0$ for every $\lambda$.
\end{enumerate}
\end{theorem}
\begin{proof}
$(1)\Leftrightarrow(2)$. Since $\nu=\Ad_Fe_3$, differentiating and using \eqref{eq:xi-brackets},
\[
\nu_y=\Ad_F[F^{-1}F_y,e_3]=-2\Ad_F(B^1\xi_1-B^2\xi_2),
\]
and differentiating once more,
\[
\nu_{xy}=-2\Ad_F\big((2A^3B^1+B^1_x)\xi_1+(2A^3B^2-B^2_x)\xi_2-(A^1B^2+A^2B^1)e_3\big).
\]
Because $[\xi_1,e_3]=-2\xi_1$ and $[\xi_2,e_3]=2\xi_2$ while $[e_3,e_3]=0$, the equation $[\nu,\nu_{xy}]=0$ is equivalent to the vanishing of the $\xi_1$- and $\xi_2$-components, that is to
\[
2A^3B^1+B^1_x=0\qquad\text{and}\qquad 2A^3B^2-B^2_x=0 .
\]
Substituting Lemma~\ref{lem:frame-coeffs}, the first reads $\varepsilon_2e^{-\omega/2}(R_x-\tfrac12\omega_xR)+\tfrac12\omega_x\varepsilon_2e^{-\omega/2}R=\varepsilon_2e^{-\omega/2}R_x=0$, i.e.\ $R_x=0$, and the second reads $H_x=0$ in the same way. By the Codazzi equations of Proposition~\ref{prop:GC}, $R_x=0$ is equivalent to $H_y=0$. Hence $\nu$ is harmonic if and only if $H_x=H_y=0$.

$(2)\Leftrightarrow(3)$. Expanding $d\alpha^\lambda+\tfrac12[\alpha^\lambda,\alpha^\lambda]=0$ in powers of $\lambda$ using \eqref{eq:xi-brackets} produces three groups of equations. The coefficient of $\lambda$ gives $A^1_y+2A^1B^3=0$ and the coefficient of $\lambda^{-1}$ gives $B^2_x-2A^3B^2=0$; by Lemma~\ref{lem:frame-coeffs} these are equivalent to $H_y=0$ and $H_x=0$ respectively, since the terms involving $\omega_x,\omega_y$ cancel exactly as above. The $\lambda$-independent part is
\[
(2A^3B^1+B^1_x)\xi_1+(2A^2B^3-A^2_y)\xi_2+(A^1B^2-A^2B^1+B^3_x-A^3_y)e_3=0 .
\]
Its $e_3$-component is the Gauss equation of Proposition~\ref{prop:GC}, hence holds identically; its $\xi_1$- and $\xi_2$-components are $R_x=0$ and $Q_y=0$, which by Codazzi are again $H_y=0$ and $H_x=0$. So the whole system holds for every $\lambda$ if and only if $H$ is constant.
\end{proof}

When $f$ is CMC, Theorem~\ref{thm:cmc-harmonic} allows one, on a simply connected domain, to integrate $\alpha^\lambda$ for every $\lambda$ simultaneously and obtain an \emph{extended frame} $\hat F:\Omega\to G$ with
\[
\hat F^{-1}d\hat F=\alpha^\lambda,\qquad \hat F|_{\lambda=1}=F .
\]
In matrix form, using $\xi_1,\xi_2$ as the off-diagonal basis vectors, the Maurer--Cartan form of an extended frame obtained in this way has the \emph{admissible} shape
\begin{equation}
\label{eq:admissible}
\hat F^{-1}d\hat F=\underbrace{\big(A^2\xi_2+A^1\lambda\,\xi_1\big)}_{\text{off-diagonal}}dx+\underbrace{\big(B^1\xi_1+B^2\lambda^{-1}\xi_2\big)}_{\text{off-diagonal}}dy+\underbrace{\big(A^3dx+B^3dy\big)e_3}_{=:\alpha_0,\ \text{diagonal}},
\end{equation}
where $\alpha_0$ is independent of $\lambda$. Thus the diagonal part $\alpha_0$ carries no $\lambda$-dependence, while the off-diagonal part carries all of it, together with the $\lambda$-independent terms $A^2\xi_2\,dx$ and $B^1\xi_1\,dy$; by Lemma~\ref{lem:frame-coeffs}, $\alpha_0$ has the generally nonzero coefficients $A^3=\tfrac14\omega_x$ and $B^3=-\tfrac14\omega_y$. We call the extended frame \emph{regular} if
\[
A^1\ne0\qquad\text{and}\qquad B^2\ne0
\]
everywhere; by Lemma~\ref{lem:frame-coeffs} this says $H\ne\pm1$, and it is exactly the condition under which the construction of the next subsection produces an immersion. By Proposition~\ref{prop:dpw}, every admissible extended frame arises from a potential pair through a Birkhoff factorization, so the construction of timelike CMC surfaces in $\Hyp$ is reduced to the choice of a potential pair whose factorization yields a frame of the shape~\eqref{eq:admissible}.

\subsection{The Sym--Bobenko formula}
\label{sec:cmc-symbob}

\begin{proposition}[{after~\cite[Thm.~4(2)]{Lee06}; cf.~\cite[\S3.3]{BIK14}}]
\label{prop:sym-bobenko}
Let $\hat F:\Omega\to G$ be a regular admissible extended frame and write $F=\hat F|_{\lambda=1}$. For every $\mu\ne0,1$ the map
\[
f^\mu:=\hat F|_{\lambda=\mu}\,\hat F|_{\lambda=1}^{-1}:\Omega\longrightarrow\Hyp
\]
is a timelike immersion with unit normal $N^\mu=\hat F|_{\lambda=\mu}\,e_3\,\hat F|_{\lambda=1}^{-1}$, Gauss map $\nu=\Ad_Fe_3$, and constant mean curvature
\[
H=\frac{\mu+1}{\mu-1} .
\]
Moreover
\[
(f^\mu)^{-1}f^\mu_x=A^1(\mu-1)\,\Ad_F\xi_1,\qquad (f^\mu)^{-1}f^\mu_y=B^2(\mu^{-1}-1)\,\Ad_F\xi_2 .
\]
\end{proposition}
\begin{proof}
Write $f=f^\mu$, $F=\hat F|_{\lambda=1}$ and $\mathcal G=\hat F|_{\lambda=\mu}$, so that $f=\mathcal GF^{-1}$ and
\[
u:=f^{-1}f_x=F\mathcal G^{-1}\big(\mathcal G_xF^{-1}-\mathcal GF^{-1}F_xF^{-1}\big)=\Ad_F\big(\mathcal G^{-1}\mathcal G_x-F^{-1}F_x\big).
\]
By \eqref{eq:admissible}, $\hat F^{-1}\hat F_x=A^1\lambda\,\xi_1+A^2\xi_2+A^3e_3$ where $A^1,A^2,A^3$ are functions of $(x,y)$ alone, the entire $\lambda$-dependence being the displayed factor. Evaluating at $\lambda=\mu$ and at $\lambda=1$ and subtracting therefore cancels the $A^2\xi_2$ and $A^3e_3$ terms, leaving $\mathcal G^{-1}\mathcal G_x-F^{-1}F_x=A^1(\mu-1)\xi_1$, so that
\[
u=p\,\Ad_F\xi_1,\qquad p:=A^1(\mu-1),
\]
and the same computation on the $y$-side gives
\[
v:=f^{-1}f_y=q\,\Ad_F\xi_2,\qquad q:=B^2(\mu^{-1}-1).
\]
This proves the two displayed frame identities.

\emph{Nondegeneracy and normal.} Since $\xi_1,\xi_2$ are null and $\Ad_F$ is an isometry, $\langle f_x,f_x\rangle=p^2\langle\xi_1,\xi_1\rangle=0=\langle f_y,f_y\rangle$, so $(x,y)$ are null coordinates for $f$ as well, and
\[
\langle f_x,f_y\rangle=pq\,\langle\xi_1,\xi_2\rangle=\tfrac12pq=\tfrac12A^1B^2(\mu-1)(\mu^{-1}-1),
\]
which is nonzero because $\mu\ne0,1$ and $A^1,B^2\ne0$ by regularity; hence $f$ is a nondegenerate timelike immersion. Setting $N:=f\Ad_Fe_3=\mathcal Ge_3F^{-1}$ we get $\langle N,f_x\rangle=p\langle e_3,\xi_1\rangle=0$, $\langle N,f_y\rangle=q\langle e_3,\xi_2\rangle=0$ and $\langle N,N\rangle=\langle e_3,e_3\rangle=1$, so $N$ is a unit normal for $f$ and the Gauss map is $f^{-1}N=\Ad_Fe_3$, unchanged from that of $F$.

\emph{Mean curvature.} With $2\langle f_x,f_y\rangle=pq$ playing the role of $\varepsilon e^{\omega}$, the definition of $H$ reads $H=2\langle f_{xy},N\rangle/(pq)$, so it suffices to compute $\langle f_{xy},N\rangle$. Expanding as in \eqref{eq:two-expansions},
\[
v_x=(\partial_xq)\Ad_F\xi_2+q\,\Ad_F[F^{-1}F_x,\xi_2]=\Ad_F\big((\partial_xq-2qA^3)\xi_2+qA^1e_3\big),
\]
\[
uv=pq\,\Ad_F(\xi_1\xi_2)=\tfrac{pq}2\big(\Id+\Ad_Fe_3\big),
\]
and pairing $f^{-1}f_{xy}=v_x+uv$ with $N$, the $\xi_2$-term and the $\Id$-term drop out, leaving
\[
\langle f_{xy},N\rangle=qA^1+\tfrac{pq}2=qA^1\Big(1+\frac p{2A^1}\Big)=qA^1\Big(1+\frac{\mu-1}2\Big)=qA^1\cdot\frac{\mu+1}2 .
\]
Dividing by $\tfrac12pq$ and cancelling $q$,
\[
H=\frac{2\langle f_{xy},N\rangle}{pq}=\frac{A^1(\mu+1)}{p}=\frac{A^1(\mu+1)}{A^1(\mu-1)}=\frac{\mu+1}{\mu-1},
\]
a constant. Expanding instead as $u_y+vu$ gives $\langle f_{xy},N\rangle=-pB^2-\tfrac{pq}2$ and hence, independently, the same value of $H$, as it must.
\end{proof}

\begin{remark}
\label{rem:H-range}
The sign of $\mu$ decides the range of the mean curvature. As $\mu$ ranges over $(0,1)\cup(1,\infty)$ the mean curvature $H=(\mu+1)/(\mu-1)$ ranges over $(-\infty,-1)\cup(1,\infty)$, so $|H|>1$; as $\mu$ ranges over $(-\infty,0)$ it ranges over $(-1,1)$, so $|H|<1$. Both halves occur in the associated family of a single admissible frame, since Proposition~\ref{prop:sym-bobenko} requires only $\mu\ne0,1$; the value $H$ of the immersion the data came from is recovered at $\mu=(H+1)/(H-1)$, which is negative exactly when $|H|<1$. What is excluded is the pair of values $H=\pm1$ themselves, which are the degenerate limits $\mu\to\infty$ and $\mu\to0$. The excluded value $\mu=1$ is excluded for the trivial reason that $f^1$ is constant.
\end{remark}

\subsection{Examples}
\label{sec:cmc-example}

We give two examples in closed form. The first is the simplest instance of the construction: its coefficients are constant, the extended frame is a single matrix exponential, and its whole associated family is flat, which by Proposition~\ref{prop:constant-potential-flat} is forced. The second has nonconstant coefficients and nonconstant Gaussian curvature, and its extended frame is nevertheless explicit.

Take $\omega\equiv0$, $\varepsilon=\varepsilon_1=\varepsilon_2=1$ and the constants
\[
H=0,\qquad Q=\tfrac12,\qquad R=-\tfrac12 .
\]
All derivatives in the Gauss--Codazzi system of Proposition~\ref{prop:GC} vanish, and the Gauss equation reduces to $\tfrac12(H^2-1)-2QR=-\tfrac12+\tfrac12=0$, so the data are admissible. By Lemma~\ref{lem:frame-coeffs},
\[
A^1=A^2=B^1=B^2=-\tfrac12,\qquad A^3=B^3=0,
\]
so the frame is regular. Since $A^3=B^3=0$, the matrices
\[
U^\lambda=-\tfrac12\lambda\,\xi_1-\tfrac12\xi_2,\qquad V^\lambda=-\tfrac12\xi_1-\tfrac1{2\lambda}\xi_2
\]
are constant in $(x,y)$, and they commute: their bracket has only an $e_3$-component, equal to $(A^1B^2-A^2B^1)e_3=(\tfrac14-\tfrac14)e_3=0$. The extended frame therefore integrates in closed form as $\hat F(x,y,\lambda)=\exp(xU^\lambda+yV^\lambda)$. Writing $M=xU^\lambda+yV^\lambda=p\,\xi_1+q\,\xi_2$ with
\[
p=-\tfrac12(\lambda x+y),\qquad q=-\tfrac1{2\lambda}(\lambda x+y),
\]
we have $M^2=pq\,\Id$ by \eqref{eq:xi-products}, so with $k:=\dfrac{\lambda x+y}{2\sqrt\lambda}$ for $\lambda>0$ (so that $pq=k^2$),
\[
\hat F(x,y,\lambda)=\cosh(k)\,\Id+\frac{\sinh(k)}k\,M=\begin{pmatrix}\cosh k & -\lambda^{-1/2}\sinh k\\ -\lambda^{1/2}\sinh k & \cosh k\end{pmatrix}\in\SLR .
\]
At $\lambda=1$ this is $F(x,y)=\begin{pmatrix}\cosh t&-\sinh t\\-\sinh t&\cosh t\end{pmatrix}$ with $t=\tfrac12(x+y)$, and Proposition~\ref{prop:sym-bobenko} produces, for each $\mu\ne0,1$, the timelike CMC immersion
\[
f^\mu=\hat F(\cdot,\cdot,\mu)F^{-1},\qquad H=\frac{\mu+1}{\mu-1}.
\]
For $\mu<0$ the formula for $\hat F$ still holds, $k$ being then imaginary and $\cosh,\sinh$ becoming $\cos,\sin$; the surface the data came from, with $H=0$, is recovered at $\mu=-1$, in accordance with Remark~\ref{rem:H-range}.

The entire associated family is flat. Indeed, computing the invariants of $f^\mu$ from the proof of Proposition~\ref{prop:sym-bobenko} --- where $p=A^1(\mu-1)$ and $q=B^2(\mu^{-1}-1)$ were introduced, $A^i,B^i$ being the frame coefficients of Lemma~\ref{lem:frame-coeffs} --- one finds
\[
Q^\mu=-pA^2=-\tfrac14(\mu-1),\qquad R^\mu=qB^1=\tfrac14(\mu^{-1}-1),\qquad \varepsilon^\mu e^{\omega^\mu}=pq=-\frac{(\mu-1)^2}{4\mu},
\]
where $Q^\mu,R^\mu,\omega^\mu,\varepsilon^\mu$ denote the invariants of $f^\mu$ itself, in the sense of \S\ref{sec:cmc-frame} applied to that immersion, so that by \eqref{eq:K-cmc},
\[
K^\mu=(H^\mu)^2-1-4e^{-2\omega^\mu}Q^\mu R^\mu=\frac{4\mu}{(\mu-1)^2}-\frac{4\mu}{(\mu-1)^2}=0 .
\]
Consequently this family lies exactly on the degenerate locus $K=0$ excluded in Theorem~\ref{thm:bridge} below, and by Proposition~\ref{prop:degenerate-locus} its Gauss map is nowhere an immersion. It does, on the other hand, satisfy the hypothesis of Corollary~\ref{cor:bridge-cmc} at every point, since $A^2B^1=\tfrac14\ne0$; see Remark~\ref{rem:three-values}.

This is not an accident of the constants chosen, and it explains why examples to which Section~\ref{sec:bridge} applies require more work.

\begin{proposition}
\label{prop:constant-potential-flat}
If the coefficients $A^1,\dots,B^3$ of an adapted frame are all constant, the immersion is flat. In particular no extended frame with constant coefficients, such as the one above, satisfies the hypothesis of Theorem~\ref{thm:bridge}.
\end{proposition}
\begin{proof}
If $A^3$ and $B^3$ are constant then $\omega_x$ and $\omega_y$ are constant by Lemma~\ref{lem:frame-coeffs}, so $\omega$ is affine and $\omega_{xy}=0$; by~\eqref{eq:K-cmc}, $K=0$.
\end{proof}

Examples to which Section~\ref{sec:bridge} applies therefore need nonconstant coefficients. The following one is nevertheless explicit.

\begin{example}[an explicit family with nonconstant curvature]
\label{ex:sinh-gordon}
Take $\varepsilon=\varepsilon_1=\varepsilon_2=1$ and constants $H>1$ and $Q,R>0$. The Codazzi equations hold automatically, since $H,Q,R$ are constant, and the Gauss equation reads
\[
\omega_{xy}=-\tfrac12(H^2-1)\,e^{\omega}+2QR\,e^{-\omega} .
\]
Put $e^{\omega_0}=2\sqrt{QR/(H^2-1)}$ and $k=\big(4QR(H^2-1)\big)^{1/4}$. Then $\omega=\omega_0+v$ satisfies the Gauss equation exactly when $v_{xy}=-k^2\sinh v$. A solution depending only on $x-y$, $v=V(k(x-y))$, must satisfy $V''=\sinh V$, and the separatrix solution $V(\varsigma)=4\operatorname{artanh}(e^\varsigma)$, $\varsigma<0$, is elementary: $V'=-2/\sinh\varsigma$ and $V''=2\cosh\varsigma/\sinh^2\varsigma=\sinh V$. Thus on the half-plane $\{x<y\}$, writing $T=e^{k(x-y)}\in(0,1)$,
\begin{equation}
\label{eq:sg-solution}
e^{\omega}=2\sqrt{\frac{QR}{H^2-1}}\,\Big(\frac{1+T}{1-T}\Big)^2
\end{equation}
solves the Gauss equation. By~\eqref{eq:K-cmc}, the Gaussian curvature of the surface determined by these data is
\[
K=H^2-1-4e^{-2\omega}QR=(H^2-1)\Big(1-\Big(\frac{1-T}{1+T}\Big)^4\Big)=\frac{8(H^2-1)\,T(1+T^2)}{(1+T)^4} .
\]
It takes values in $(0,H^2-1)$ and is not constant; in particular $\omega_{xy}\ne0$ on all of $\{x<y\}$.

By Theorem~\ref{thm:cmc-harmonic} the data determine an extended frame $\hat F$, unique up to left translation once an initial value is fixed. By Lemma~\ref{lem:frame-coeffs} its coefficients are
\[
A^1=\tfrac12e^{\omega/2}(H-1),\quad A^2=-e^{-\omega/2}Q,\quad A^3=\tfrac14\omega_x,
\]
\[
B^1=e^{-\omega/2}R,\quad B^2=-\tfrac12e^{\omega/2}(H+1),\quad B^3=-\tfrac14\omega_y ,
\]
all elementary functions of $T$. They are nowhere zero, so the frame is regular. Moreover $A^1B^2-A^2B^1=\tfrac12\omega_{xy}\ne0$, so the Gauss map $\nu=\Ad_Fe_3$ is an immersion on all of $\{x<y\}$. The frame itself can be written in closed form. Its coefficients depend on $x-y$ alone, so $\hat F=\exp\big((x+y)Y\big)F_0(x-y)$ with $Y$ constant, and $F_0$ solves a linear ordinary differential equation in $u=x-y$. That equation has solutions of the form $e^{\pm i\varpi u}(\alpha\tau+\beta)/\sqrt\tau$, where $\tau=(1+T)/(1-T)$. Explicitly, for $\lambda>0$ put
\[
\kappa=\frac{\sqrt{2\lambda(H-1)Q}}{k},\qquad \varpi=\frac{k(\kappa^2-1)}{4\kappa},\qquad \hat\varpi=\frac{k(\kappa^2+1)}{4\kappa},\qquad \sigma_0=\frac{\kappa k\,e^{\omega_0/2}}{2Q} .
\]
Then, up to left translation,
\[
\hat F=\frac{\exp\big(\hat\varpi(x+y)e_1\big)}{\sqrt{\sigma_0(1+\kappa^2)\,\tau}}
\begin{pmatrix}\sigma_0\big(\kappa\tau\cos\varpi u+\sin\varpi u\big)&\tau\cos\varpi u-\kappa\sin\varpi u\\ \sigma_0\big(\kappa\tau\sin\varpi u-\cos\varpi u\big)&\tau\sin\varpi u+\kappa\cos\varpi u\end{pmatrix},
\]
with $\exp(ce_1)=\cos c\,\Id+\sin c\,e_1$. A direct computation shows that $\det\hat F=1$ and that $\hat F^{-1}d\hat F$ is the one-form~\eqref{eq:admissible} with the coefficients above.

By Proposition~\ref{prop:sym-bobenko}, each $f^\mu=\hat F|_{\lambda=\mu}\hat F|_{\lambda=1}^{-1}$ with $\mu>0$, $\mu\ne1$, is then an explicit timelike immersion of constant mean curvature $(\mu+1)/(\mu-1)$. Its Gaussian curvature is
\[
K^\mu=(H^\mu)^2-1+\frac{4\big(H^2-1-K\big)}{(\mu-1)(\mu^{-1}-1)(H^2-1)}
=\frac{4\mu}{(\mu-1)^2\,(H^2-1)}\,K ,
\]
so that $K^\mu$ is not constant. The second form of this identity follows from the first by $(H^\mu)^2-1=4\mu/(\mu-1)^2$ and $(\mu-1)(\mu^{-1}-1)=-(\mu-1)^2/\mu$. It holds for every admissible frame, not only for these data. The factor in front of $K$ is positive for every $\mu>0$, so $K^\mu$ and $K$ have the same zeros and the same sign throughout the associated family, which gives the $\mu$-independence used in Remark~\ref{rem:bridge-parallel} without passing through Proposition~\ref{prop:degenerate-locus}. Taking $R\equiv0$ instead reduces the Gauss equation to Liouville's equation and gives examples in which every $f^\mu$ is isoparametric; the family above avoids this.
\end{example}

\section{Timelike surfaces of constant Gaussian curvature}
\label{sec:cgc}

We now change the conformal structure on the domain. Let $f:\Omega\to\Hyp$ be a timelike immersion whose extrinsic curvature $K+1$ is positive. Then $\det II=(K+1)\det I<0$ by~\eqref{eq:gauss-general}, since $\det I<0$ for a timelike immersion, so the second fundamental form is again a Lorentzian metric on $\Omega$; this is the metric with respect to which harmonicity of the Gauss map is now measured. Its null coordinates are the asymptotic coordinates of $f$, and in them the harmonic map equation is again $[\nu,\nu_{xy}]=0$.

The two hypotheses just made --- positive extrinsic curvature, and, below, that $\nu$ be an immersion --- are genuine restrictions rather than technical conveniences. By Proposition~\ref{prop:degenerate-locus} the second one excludes exactly the flat points of $f$, and flat timelike surfaces exist in abundance. Here is one. Consider
\[
f(x,y)=(c_1\cosh(x+y),\,c_2\sinh(x-y),\,c_1\sinh(x+y),\,c_2\cosh(x-y)),\qquad c_1^2-c_2^2=1,
\]
with normal $N=(c_2\cosh(x+y),c_1\sinh(x-y),c_2\sinh(x+y),c_1\cosh(x-y))$. A direct computation gives
\[
I=dx^2+2(c_1^2+c_2^2)\,dx\,dy+dy^2,\qquad II=-4c_1c_2\,dx\,dy,
\]
so $(x,y)$ are asymptotic coordinates, while in the identification~\eqref{eq:identification} the Gauss map becomes
\[
\nu=\begin{pmatrix}-\cosh 2x&-\sinh 2x\\ \sinh 2x&\cosh 2x\end{pmatrix} .
\]
This surface is flat. Indeed $c_1^2-c_2^2=1$ gives $(c_1^2+c_2^2)^2-1=4c_1^2c_2^2$, so
\[
\det I=1-(c_1^2+c_2^2)^2=-4c_1^2c_2^2=-\big(2c_1c_2\big)^2=\det II ,
\]
whence $K+1=\det II/\det I=1$ and $K\equiv0$. Consistently with Proposition~\ref{prop:degenerate-locus}, its Gauss map is not an immersion: $\nu$ satisfies $[\nu,\nu_{xy}]=0$, so it is harmonic, but $\nu_y\equiv0$. The example therefore shows both that harmonicity of the Gauss map alone does not force $\nu$ to be an immersion, so that hypothesis has to be imposed separately, and that the exclusion of flat surfaces is not vacuous.

\subsection{The characterization of constant Gaussian curvature}
\label{sec:cgc-char}

Assume from now on that $\nu$ is an immersion. Then $\nu_x,\nu_y$ are linearly independent and both orthogonal to $\nu$, since $\langle\nu,\nu\rangle$ is constant, so they form a basis of the Lorentzian plane $\nu^\perp$. By Proposition~\ref{prop:bracket-identities}(i), $[\nu_x,\nu_y]$ is orthogonal to both, hence proportional to $\nu$:
\begin{equation}
\label{eq:kappa}
[\nu_x,\nu_y]=\kappa\,\nu,\qquad \kappa=\langle[\nu_x,\nu_y],\nu\rangle .
\end{equation}
Here $\kappa\ne0$: an element of $\sltwo$ has $1$-dimensional centralizer, so $[\nu_x,\nu_y]=0$ would force $\nu_x$ and $\nu_y$ to be proportional. The tangent plane of $f$, left-translated, is $\nu^\perp$, and by Proposition~\ref{prop:bracket-identities}(v) the map $\operatorname{ad}_\nu$ is an isomorphism of $\nu^\perp$, so $\{[\nu,\nu_x],[\nu,\nu_y]\}$ is again a basis of it. We may therefore write
\begin{equation}
\label{eq:cgc-ansatz-general}
f^{-1}f_x=\tfrac{r}2[\nu,\nu_x]+c\,[\nu,\nu_y],\qquad
f^{-1}f_y=d\,[\nu,\nu_x]+\tfrac{s}2[\nu,\nu_y]
\end{equation}
for functions $r,s,c,d$ on $\Omega$.

\begin{lemma}
\label{lem:cgc-diagonal}
In asymptotic coordinates, $c=d=0$; that is,
\begin{equation}
\label{eq:cgc-ansatz}
f^{-1}f_x=\tfrac{r}2[\nu,\nu_x],\qquad f^{-1}f_y=\tfrac{s}2[\nu,\nu_y] .
\end{equation}
\end{lemma}
\begin{proof}
By definition, $(x,y)$ asymptotic means that $II$ is off-diagonal, $L=\mathcal N=0$. With $u=f^{-1}f_x$ as in~\eqref{eq:cgc-ansatz-general},
\[
u_x=\tfrac{r_x}2[\nu,\nu_x]+\tfrac r2[\nu,\nu_{xx}]+c_x[\nu,\nu_y]+c[\nu_x,\nu_y]+c[\nu,\nu_{xy}],
\]
where we used $[\nu_x,\nu_x]=0$. Every term of the form $[\nu,\,\cdot\,]$ is orthogonal to $\nu$ by Proposition~\ref{prop:bracket-identities}(i), so
\[
0=L=\langle u_x,\nu\rangle=c\,\langle[\nu_x,\nu_y],\nu\rangle=c\,\kappa ,
\]
and $\kappa\ne0$ by~\eqref{eq:kappa}, whence $c=0$. The same computation applied to $\mathcal N=\langle v_y,\nu\rangle$ gives $d=0$.
\end{proof}

\begin{lemma}
\label{lem:cgc-compat}
The functions $r,s$ of~\eqref{eq:cgc-ansatz} satisfy
\begin{equation}
\label{eq:cgc-compat}
r_y[\nu,\nu_x]-s_x[\nu,\nu_y]+(r-s)[\nu,\nu_{xy}]=0
\qquad\text{and}\qquad
r+s-2rs=0 .
\end{equation}
\end{lemma}
\begin{proof}
The compatibility condition $\partial_y(f^{-1}f_x)-\partial_x(f^{-1}f_y)=[f^{-1}f_x,f^{-1}f_y]$ --- which is just the Maurer--Cartan equation for $f^{-1}df$, or equivalently the equality of mixed partials --- gives, with $u=\tfrac r2[\nu,\nu_x]$ and $v=\tfrac s2[\nu,\nu_y]$,
\[
u_y-v_x=\tfrac{r_y}2[\nu,\nu_x]-\tfrac{s_x}2[\nu,\nu_y]+\tfrac{r-s}2[\nu,\nu_{xy}]-\tfrac{r+s}2[\nu_x,\nu_y],
\]
using $[\nu_y,\nu_x]=-[\nu_x,\nu_y]$, while by Proposition~\ref{prop:bracket-identities}(iii) with $Z=\nu$ and $\langle\nu,\nu\rangle=1$,
\[
[u,v]=\tfrac{rs}4\big[[\nu,\nu_x],[\nu,\nu_y]\big]=-rs\,[\nu_x,\nu_y].
\]
Equating and multiplying by $2$,
\[
r_y[\nu,\nu_x]-s_x[\nu,\nu_y]+(r-s)[\nu,\nu_{xy}]-(r+s-2rs)[\nu_x,\nu_y]=0 .
\]
The first three terms are orthogonal to $\nu$ by Proposition~\ref{prop:bracket-identities}(i), hence tangential, whereas $[\nu_x,\nu_y]=\kappa\nu$ is normal; the two groups therefore vanish separately, which is~\eqref{eq:cgc-compat}.
\end{proof}

\begin{lemma}
\label{lem:cgc-curvature}
With the notation above, the extrinsic curvature and the mean curvature of $f$ are
\[
K+1=\left(\frac{1-r}{r}\right)^2,\qquad H=-2\,\frac{1-r}{r}\,\frac{\langle\nu_x,\nu_y\rangle}{\kappa} .
\]
\end{lemma}
\begin{proof}
By Proposition~\ref{prop:bracket-identities}(iv) with $Z=\nu$,
\[
E=\tfrac{r^2}4\big\langle[\nu,\nu_x],[\nu,\nu_x]\big\rangle=-r^2\langle\nu_x,\nu_x\rangle,\quad
\mathcal F=-rs\langle\nu_x,\nu_y\rangle,\quad
G=-s^2\langle\nu_y,\nu_y\rangle .
\]
Since $L=\mathcal N=0$, only the mixed coefficient of $II$ remains; using Proposition~\ref{prop:bracket-identities}(i) to discard every term of the form $[\nu,\,\cdot\,]$ and Proposition~\ref{prop:bracket-identities}(iii) for the bracket term,
\[
M=\Big\langle v_x+\tfrac12[u,v],\nu\Big\rangle
=\Big\langle \tfrac s2[\nu_x,\nu_y]-\tfrac{rs}2[\nu_x,\nu_y],\nu\Big\rangle
=\tfrac s2(1-r)\,\kappa .
\]
Hence
\[
K+1=\frac{\det II}{\det I}=\frac{-\tfrac{s^2}4(1-r)^2\kappa^2}{r^2s^2\big(\langle\nu_x,\nu_x\rangle\langle\nu_y,\nu_y\rangle-\langle\nu_x,\nu_y\rangle^2\big)} .
\]
Two identities simplify this. First, since $[\nu_x,\nu_y]=\kappa\nu$ by~\eqref{eq:kappa} and $\langle\nu,\nu\rangle=1$,
\[
\kappa^2=\big\langle[\nu_x,\nu_y],[\nu_x,\nu_y]\big\rangle .
\]
Second, Proposition~\ref{prop:bracket-identities}(ii) applied to $X=\nu_x$, $Y=\nu_y$ gives
\[
\langle\nu_x,\nu_x\rangle\langle\nu_y,\nu_y\rangle-\langle\nu_x,\nu_y\rangle^2=-\tfrac14\big\langle[\nu_x,\nu_y],[\nu_x,\nu_y]\big\rangle .
\]
Substituting both, the factor $\langle[\nu_x,\nu_y],[\nu_x,\nu_y]\rangle$ cancels and
\[
K+1=\frac{-\tfrac{s^2}4(1-r)^2}{-\tfrac14r^2s^2}=\left(\frac{1-r}{r}\right)^2 .
\]
In particular $\det I=-\tfrac14r^2s^2\kappa^2$. Since $L=\mathcal N=0$, the shape operator $S=-II\,I^{-1}$ has trace $2M\mathcal F/\det I$, so
\[
H=-\frac{M\mathcal F}{\det I}=\frac{\tfrac s2(1-r)\kappa\cdot rs\langle\nu_x,\nu_y\rangle}{-\tfrac14r^2s^2\kappa^2}=-2\,\frac{1-r}{r}\,\frac{\langle\nu_x,\nu_y\rangle}{\kappa}. \qedhere
\]
\end{proof}

\begin{theorem}[{cf.~\cite[Thm.~3.1]{BIK14} in $S^3$}]
\label{thm:cgc-harmonic}
Let $f:\Omega\to\Hyp$ be a timelike immersion with positive extrinsic curvature and Gauss map $\nu$, and assume that $\nu$ is an immersion. Then the following are equivalent:
\begin{enumerate}\itemsep=0pt
\item[\rm(1)] $\nu$ is harmonic with respect to the Lorentzian structure induced by the second fundamental form;
\item[\rm(2)] $f$ has constant Gaussian curvature.
\end{enumerate}
\end{theorem}
\begin{proof}
Work in asymptotic coordinates, so that Lemmas~\ref{lem:cgc-diagonal}--\ref{lem:cgc-curvature} apply.

Suppose (2) holds. By Lemma~\ref{lem:cgc-curvature}, $K$ is constant if and only if $(1-r)/r=1/r-1$ is constant --- a continuous function with constant square on the connected domain $\Omega$ --- that is if and only if $r$ is constant. Then the second equation of~\eqref{eq:cgc-compat}, in which $r=\tfrac12$ is impossible since it would read $\tfrac12=0$, forces $s=r/(2r-1)$ to be constant as well, and moreover $r\ne s$: were $r=s$, that equation would give $2r-2r^2=0$, hence $r\in\{0,1\}$, both excluded ($r=0$ contradicts~\eqref{eq:cgc-ansatz} defining an immersion, and $r=1$ gives $K+1=0$, contradicting positivity of the extrinsic curvature). With $r,s$ constant the first equation of~\eqref{eq:cgc-compat} becomes $(r-s)[\nu,\nu_{xy}]=0$, and since $r\ne s$ we conclude $[\nu,\nu_{xy}]=0$, which is (1).

Conversely, suppose (1) holds, so $[\nu,\nu_{xy}]=0$. The first equation of~\eqref{eq:cgc-compat} reduces to $r_y[\nu,\nu_x]=s_x[\nu,\nu_y]$, and $[\nu,\nu_x]$, $[\nu,\nu_y]$ are linearly independent by Proposition~\ref{prop:bracket-identities}(v), so $r_y=s_x=0$. Differentiating $r+s-2rs=0$ in $x$ gives $r_x(1-2s)=0$, and in $y$ gives $s_y(1-2r)=0$. If $s=\tfrac12$ then $r+\tfrac12-r=\tfrac12\ne0$, contradicting~\eqref{eq:cgc-compat}; hence $1-2s\ne0$ and $r_x=0$, so $r$ is constant, and likewise $s$. By Lemma~\ref{lem:cgc-curvature}, $K$ is constant, which is (2).
\end{proof}

\subsection{From a harmonic Gauss map back to the surface}
\label{sec:cgc-converse}

Theorem~\ref{thm:cgc-harmonic} turns the construction of timelike CGC surfaces into the construction of harmonic maps into $\Sph$, provided one can go back from $\nu$ to $f$. From $\nu$ alone the passage back is an integration, in contrast with the CMC case of Proposition~\ref{prop:sym-bobenko}, where the surface is read off algebraically from the frame. The contrast is one of starting point rather than of substance: once the extended frame of \S\ref{sec:cgc-cauchy} has been built, the surface is recovered by the same algebraic evaluation here too (Proposition~\ref{prop:cgc-sym}).

\begin{proposition}
\label{prop:cgc-converse}
Let $\Omega\subset\mathbb R^{1,1}$ be simply connected and let $\nu:\Omega\to\Sph$ be a harmonic immersion, where $\Omega$ carries a Lorentz structure with null coordinates $(x,y)$. Then for every $\rho\in\mathbb R$ with $\rho\ne0,\pm1$ there is a timelike immersion $f:\Omega\to\Hyp$ having $\nu$ as its Gauss map and constant extrinsic curvature
\[
K+1=\rho^2 ,
\]
obtained by integrating $f^{-1}df=\alpha$, where
\begin{equation}
\label{eq:cgc-alpha}
\alpha=\frac{[\nu,\nu_x]}{2(1+\rho)}\,dx+\frac{[\nu,\nu_y]}{2(1-\rho)}\,dy .
\end{equation}
The corresponding constants in~\eqref{eq:cgc-ansatz} are $r=\dfrac1{1+\rho}$ and $s=\dfrac1{1-\rho}$.
\end{proposition}
\begin{proof}
Fix constants $r,s$ with $r+s-2rs=0$ and consider the $\sltwo$-valued one-form $\alpha=\tfrac r2[\nu,\nu_x]dx+\tfrac s2[\nu,\nu_y]dy$. Since $\Omega$ is simply connected, $\alpha=f^{-1}df$ for some $f:\Omega\to\SLR$ if and only if $\alpha$ satisfies the Maurer--Cartan equation $d\alpha+\tfrac12[\alpha,\alpha]=0$, that is $\partial_y(\tfrac r2[\nu,\nu_x])-\partial_x(\tfrac s2[\nu,\nu_y])=[\tfrac r2[\nu,\nu_x],\tfrac s2[\nu,\nu_y]]$. This is the computation of Lemma~\ref{lem:cgc-compat} read backwards: with $r,s$ constant it reduces to
\[
(r-s)[\nu,\nu_{xy}]-(r+s-2rs)[\nu_x,\nu_y]=0 ,
\]
whose first term vanishes because $\nu$ is harmonic and whose second vanishes by the choice of $r,s$. Hence $\alpha$ integrates to a map $f$, unique up to left translation, which satisfies~\eqref{eq:cgc-ansatz}; it is an immersion because $\nu$ is one and $r,s\ne0$. Moreover $(x,y)$ are automatically asymptotic for $f$: the computation of Lemma~\ref{lem:cgc-diagonal}, carried out with $c=d=0$ from the start, gives $L=\langle u_x,\nu\rangle=0$ and likewise $\mathcal N=0$, since every surviving term is of the form $[\nu,\,\cdot\,]$ and hence orthogonal to $\nu$. Lemma~\ref{lem:cgc-curvature} therefore applies and gives $K+1=((1-r)/r)^2$. Finally, $\langle\nu,f^{-1}f_x\rangle$ and $\langle\nu,f^{-1}f_y\rangle$ vanish by Proposition~\ref{prop:bracket-identities}(i), and $\langle\nu,\nu\rangle=1$, so $f\nu$ is a unit normal to $f$; that is, $\nu$ is the Gauss map of $f$.

It remains to match the parameters. Writing $K+1=\rho^2$, the equation $(1-r)/r=\rho$ gives $r=1/(1+\rho)$, and then $s$ is determined by $r+s-2rs=0$:
\[
s=\frac{r}{2r-1}=\frac{\frac1{1+\rho}}{\frac{2}{1+\rho}-1}=\frac1{1-\rho},
\]
which requires $\rho\ne1$; conversely $r=1/(1+\rho)$, $s=1/(1-\rho)$ satisfy $r+s=\frac2{1-\rho^2}=2rs$. Substituting in $\alpha$ gives~\eqref{eq:cgc-alpha}. Replacing $\rho$ by $-\rho$ interchanges $r$ and $s$ and leaves $K$ unchanged, so each curvature $K+1=\rho^2$ is realized by the two surfaces with parameters $\pm\rho$, which have the same Gauss map and differ in which asymptotic direction carries the factor $1/(1+|\rho|)$.
\end{proof}

\begin{remark}
\label{rem:rho-branches}
The parameter $\rho$ ranges over $\mathbb R\setminus\{0,\pm1\}$, and the curvature depends only on $|\rho|$: $|\rho|>1$ gives $K=\rho^2-1>0$, while $0<|\rho|<1$ gives $-1<K<0$. Both branches occur in Proposition~\ref{prop:cgc-converse}, from the same harmonic map $\nu$, and both are reached from a constant mean curvature frame in Section~\ref{sec:bridge}. The excluded values $\rho=\pm1$ are exactly $K=0$, which by Proposition~\ref{prop:degenerate-locus} is where the Gauss map ceases to be an immersion; so the two standing hypotheses of this section, $K+1>0$ and $\nu$ immersive, are together equivalent to $\rho\ne0,\pm1$, and nothing is lost by excluding those values here.
\end{remark}

\subsection{Extended frames and the Cauchy problem}
\label{sec:cgc-cauchy}

We now put the harmonic map $\nu$ into loop group form and use it to solve the Cauchy problem. Since $\Sph\cong\SLR/\mathrm{Stab}(e_3)$, a \emph{frame} of $\nu:\Omega\to\Sph$ is a map $F:\Omega\to\SLR$ with $\nu=\Ad_Fe_3$; it is determined by $\nu$ up to $F\mapsto\pm F\exp(\varphi e_3)$ with $\varphi$ a function. Write $\mathfrak k=\mathbb R e_3$ and $\mathfrak p=\operatorname{span}(e_1,e_2)=e_3^\perp$, and split
\begin{equation}
\label{eq:frame-split}
F^{-1}F_x=\beta_0+\beta_1,\qquad F^{-1}F_y=\gamma_0+\gamma_{-1},\qquad \beta_0,\gamma_0\in\mathfrak k,\quad \beta_1,\gamma_{-1}\in\mathfrak p .
\end{equation}
Since $[\mathfrak p,e_3]\subset\mathfrak p$ and $[e_3,[e_3,\beta]]=4\beta$ for $\beta\in\mathfrak p$ by Proposition~\ref{prop:bracket-identities}(v), differentiating $\nu=\Ad_Fe_3$ gives
\begin{equation}
\label{eq:nu-derivs}
\begin{aligned}
&\nu_x=\Ad_F[\beta_1,e_3],&\qquad &\nu_y=\Ad_F[\gamma_{-1},e_3],\\
&\beta_1=-\tfrac14\big[e_3,\Ad_F^{-1}\nu_x\big],&\qquad &\gamma_{-1}=-\tfrac14\big[e_3,\Ad_F^{-1}\nu_y\big].
\end{aligned}
\end{equation}
In particular $\nu$ is an immersion exactly where $\beta_1$ and $\gamma_{-1}$ are linearly independent.

\begin{lemma}[{cf.~\cite[\S3]{BIK14}}]
\label{lem:cgc-mc}
The map $\nu$ is harmonic if and only if, for every $\lambda\in\mathbb R\setminus\{0\}$, the one-form
\begin{equation}
\label{eq:cgc-extended}
\hat\alpha^\lambda=(\beta_0+\lambda\beta_1)\,dx+(\gamma_0+\lambda^{-1}\gamma_{-1})\,dy
\end{equation}
satisfies the Maurer--Cartan equation. The family $\hat\alpha^\lambda$ satisfies $\hat\alpha^{-\lambda}=\Ad_{e_3}\hat\alpha^\lambda$, and its form is preserved by the gauge $F\mapsto F\exp(\varphi e_3)$.
\end{lemma}
\begin{proof}
With $U=\beta_0+\lambda\beta_1$ and $V=\gamma_0+\lambda^{-1}\gamma_{-1}$, the Maurer--Cartan equation $U_y-V_x=[U,V]$ splits by powers of $\lambda$, using $[\beta_0,\gamma_0]=0$, into
\[
\partial_y\beta_1=[\beta_1,\gamma_0],\qquad -\partial_x\gamma_{-1}=[\beta_0,\gamma_{-1}],\qquad \partial_y\beta_0-\partial_x\gamma_0=[\beta_1,\gamma_{-1}] .
\]
The bracket of two elements of $\mathfrak p$ lies in $\mathfrak k$, and that of an element of $\mathfrak k$ with one of $\mathfrak p$ lies in $\mathfrak p$. So the third equation is the $\mathfrak k$-part of the Maurer--Cartan equation of $F$, and the sum of the first two is its $\mathfrak p$-part. Since $F$ is a frame, the family is flat for every $\lambda$ if and only if the first equation holds. On the other hand, differentiating $\nu_x=\Ad_F[\beta_1,e_3]$ in $y$ and using the Jacobi identity together with $[\gamma_0,e_3]=0$,
\[
\Ad_F^{-1}\nu_{xy}=[\gamma_{-1},[\beta_1,e_3]]+\big[\partial_y\beta_1-[\beta_1,\gamma_0],\,e_3\big],
\]
where the first term lies in $\mathfrak k$ and the second in $\mathfrak p$. As $\operatorname{ad}_{e_3}$ is injective on $\mathfrak p$, the condition $\nu_{xy}\in\mathbb R\nu$, which is~\eqref{eq:harmonic}, is again equivalent to $\partial_y\beta_1=[\beta_1,\gamma_0]$. The last two assertions hold because $\Ad_{e_3}$ fixes $\mathfrak k$ and negates $\mathfrak p$, and $\Ad_{\exp(-\varphi e_3)}$ preserves both.
\end{proof}

When $\nu$ is harmonic and $\Omega$ is simply connected, integrating $\hat F^{-1}d\hat F=\hat\alpha^\lambda$ with $\hat F|_{\lambda=1}=F$ gives an \emph{extended frame} $\hat F$ of $\nu$. It is $\sigma$-twisted, and of the admissible shape of Proposition~\ref{prop:dpw}, with $\lambda$-degrees $\{0,1\}$ in $dx$ and $\{-1,0\}$ in $dy$. No normalization of $\nu_x$ or $\nu_y$ is needed for this: a null or timelike derivative of $\nu$ is as good as a spacelike one (Example~\ref{ex:akamine-kiyohara}).

Before turning to the Cauchy problem, we record a consequence of the extended frame which is not visible from the harmonic map alone. Proposition~\ref{prop:cgc-converse} recovers the surface from $\nu$ by an integration; once the \emph{extended} frame is at hand, the surface is read off by the same two-point evaluation that Proposition~\ref{prop:sym-bobenko} uses on the constant mean curvature side.

\begin{proposition}[{cf.~\cite[(3.9)]{BIK14} in $S^3$}]
\label{prop:cgc-sym}
Let $\nu$ be a harmonic immersion with extended frame $\hat F$, and let $\mu\in\mathbb R\setminus\{0,\pm1\}$. Then
\begin{equation}
\label{eq:cgc-sym}
g^\mu:=\hat F|_{\lambda=\mu}\,\hat F|_{\lambda=1}^{-1}
\end{equation}
is a timelike immersion with Gauss map $\nu$, for which $(x,y)$ are asymptotic coordinates, of constant extrinsic curvature
\[
K+1=\Big(\frac{1+\mu}{1-\mu}\Big)^{2};
\]
it is the surface of Proposition~\ref{prop:cgc-converse} for $\rho=\tfrac{1+\mu}{1-\mu}$. Moreover $\mu>0$ gives $|\rho|>1$, that is $K>0$, and $\mu<0$ gives $|\rho|<1$, that is $-1<K<0$.
\end{proposition}
\begin{proof}
Write $F=\hat F|_{\lambda=1}$. Differentiating~\eqref{eq:cgc-sym}, and noting that the $\lambda$-independent terms of~\eqref{eq:cgc-extended} cancel,
\[
(g^\mu)^{-1}dg^\mu=\Ad_F\big(\hat\alpha^\mu-\hat\alpha^1\big)=\Ad_F\big((\mu-1)\beta_1\,dx+(\mu^{-1}-1)\gamma_{-1}\,dy\big).
\]
By~\eqref{eq:nu-derivs} and $[e_3,[e_3,\beta]]=4\beta$ for $\beta\in\mathfrak p$, we have $[\nu,\nu_x]=\Ad_F[e_3,[\beta_1,e_3]]=-4\Ad_F\beta_1$, and likewise $[\nu,\nu_y]=-4\Ad_F\gamma_{-1}$. Hence
\[
(g^\mu)^{-1}dg^\mu=\tfrac r2[\nu,\nu_x]\,dx+\tfrac s2[\nu,\nu_y]\,dy,
\qquad r=\frac{1-\mu}2,\quad s=\frac{\mu-1}{2\mu} .
\]
This is~\eqref{eq:cgc-ansatz}, and
\[
r+s=\frac{(1-\mu)\mu+(\mu-1)}{2\mu}=-\frac{(1-\mu)^2}{2\mu}=2rs .
\]
This is the identity $r+s-2rs=0$ on which the proof of Proposition~\ref{prop:cgc-converse} rests. That proof therefore applies verbatim, $\nu$ being a harmonic immersion, and gives a timelike immersion with Gauss map $\nu$ and asymptotic coordinates $(x,y)$. Moreover $(1-r)/r=\tfrac{1+\mu}{1-\mu}$ is the value of $\rho$ for which $r=1/(1+\rho)$, so $g^\mu$ is the surface of Proposition~\ref{prop:cgc-converse} for that $\rho$, and $K+1=\rho^2$. Finally, $\mu\mapsto(1+\mu)/(1-\mu)$ maps $(0,1)$ onto $(1,\infty)$ and $(1,\infty)$ onto $(-\infty,-1)$, so $|\rho|>1$ for $\mu>0$. It maps $(-1,0)$ onto $(0,1)$ and $(-\infty,-1)$ onto $(-1,0)$, so $|\rho|<1$ for $\mu<0$.
\end{proof}

\begin{remark}
\label{rem:cgc-sym}
The formula has exactly the shape of Proposition~\ref{prop:sym-bobenko}: two evaluations of one extended frame, and no derivative in $\lambda$. The integration of~\eqref{eq:cgc-alpha} is thus not a cost intrinsic to the constant Gaussian curvature side. It is the price of starting from $\nu$ rather than from $\hat F$. The two branches of Remark~\ref{rem:rho-branches} are separated by the sign of the spectral parameter.
\end{remark}

The following theorem solves the Cauchy problem along a non-characteristic curve, which we take to be the diagonal. The problem is to find a harmonic map with prescribed values and prescribed $x$-derivative along the diagonal. The harmonic map equation is a determined hyperbolic system and the diagonal is non-characteristic for it, so existence and uniqueness of a local solution follow from the Cauchy--Kowalevski theorem for analytic data, and from the standard theory of semilinear hyperbolic equations for smooth data. What the loop group method provides is a construction. Theorem~\ref{thm:cauchy} produces the solution by the generalized d'Alembert procedure, from potentials written explicitly in terms of the Cauchy data by~\eqref{eq:cauchy-coefficients}, as in the flat case~\cite[\S1]{BS13}. The data are prescribed on a non-characteristic curve, rather than along the two characteristics as in~\cite[Cors.~4.2 and~4.3]{Xia07}.

\begin{theorem}
\label{thm:cauchy}
Let $I\subset\mathbb R$ be an open interval, let $N_0:I\to\Sph$ and $N_1:I\to\sltwo$ be smooth with $\langle N_1,N_0\rangle=0$, and let $F:I\to\SLR$ be any smooth map with $\Ad_Fe_3=N_0$. Put
\begin{equation}
\label{eq:cauchy-coefficients}
\beta_1=-\tfrac14\big[e_3,\Ad_F^{-1}N_1\big],\qquad \gamma_{-1}=-\tfrac14\big[e_3,\Ad_F^{-1}(N_0'-N_1)\big],\qquad \hat\alpha^\lambda=\big(F^{-1}F'\big)_{\mathfrak k}+\lambda\beta_1+\lambda^{-1}\gamma_{-1},
\end{equation}
functions on $I$, where $(\,\cdot\,)_{\mathfrak k}$ denotes the $e_3$-component. Then there are an open neighbourhood $\Omega$ of the diagonal in $I\times I$ and a harmonic map $\nu:\Omega\to\Sph$, for the Lorentz structure in which $(x,y)$ are null coordinates, with
\[
\nu(x,x)=N_0(x),\qquad \nu_x(x,x)=N_1(x).
\]
It is obtained by the DPW recipe from the potentials $\mathcal X(x)=\hat\alpha^\lambda(x)\,dx$ and $\mathcal Y(y)=\hat\alpha^\lambda(y)\,dy$ with $X(x_0)=Y(x_0)=F(x_0)$, and its extended frame $\hat F$ coincides with $X$ along the diagonal.
\end{theorem}
\begin{proof}
The one-form $\hat\alpha^\lambda$ has $\lambda$-support $\{-1,0,1\}$, which lies both in the range $\{k\le1\}$ required of $\mathcal X$ and in the range $\{k\ge-1\}$ required of $\mathcal Y$ in Proposition~\ref{prop:dpw}. (This is the generalized form of the DPW recipe~\cite{BD09}, in which the potentials are not required to be holomorphic in a complex parameter.) Since $X$ and $Y$ solve the same equation with the same initial value, $X(u)=Y(u)$ for all $u$. So $X^{-1}(x)Y(y)=\Id$ along the diagonal, which lies in the big cell; by openness, $X^{-1}(x)Y(y)$ stays in it on an open neighbourhood $\Omega$ of the diagonal, and we work there. Along the diagonal the Birkhoff factorization is trivial, $H_-=H_+=\Id$, so the extended frame $\hat F=XH_-$ agrees with $X$ there.

The potentials are $\sigma$-twisted, since $(F^{-1}F')_{\mathfrak k}\in\mathfrak k$ and $\beta_1,\gamma_{-1}\in\mathfrak p$. So $X$ and $Y$ lie in $F(x_0)G_\sigma$, the factors $H_\pm$ lie in $G_\sigma$, and the Maurer--Cartan form of $\hat F$ is $\sigma$-twisted. Together with Proposition~\ref{prop:dpw} this means that it has the shape~\eqref{eq:cgc-extended} for the frame $\hat F|_{\lambda=1}$ of $\nu:=\Ad_{\hat F|_{\lambda=1}}e_3$. It is flat for every $\lambda$, so $\nu$ is harmonic by Lemma~\ref{lem:cgc-mc}.

It remains to check the initial conditions. Along the diagonal $\hat F=X$, so the tangential derivative of $\hat F$ is $\hat\alpha^\lambda$, and the $dx$- and $dy$-parts are separated by their $\lambda$-degrees: there the coefficients of $\lambda$ and $\lambda^{-1}$ in $\hat F^{-1}d\hat F$ are the $\beta_1$ and $\gamma_{-1}$ of~\eqref{eq:cauchy-coefficients}. At $\lambda=1$, the $\mathfrak p$-part of $F^{-1}F'$ is $\beta_1+\gamma_{-1}$. Indeed $N_0'=\Ad_F[F^{-1}F',e_3]$ gives $(F^{-1}F')_{\mathfrak p}=-\tfrac14[e_3,\Ad_F^{-1}N_0']$ as in~\eqref{eq:nu-derivs}, which is $\beta_1+\gamma_{-1}$. So $\hat\alpha^1=F^{-1}F'$, and $X|_{\lambda=1}=F$ along the diagonal, both solving the same equation with the same initial value. Hence $\nu(u,u)=\Ad_{F(u)}e_3=N_0(u)$. Finally, $m:=\Ad_F^{-1}N_1\in\mathfrak p$ because $\langle N_1,N_0\rangle=0$, and by~\eqref{eq:nu-derivs}
\[
\nu_x(u,u)=\Ad_F[\beta_1,e_3]=\Ad_F\big(\tfrac14[e_3,[e_3,m]]\big)=\Ad_Fm=N_1(u). \qedhere
\]
\end{proof}

The choice of the frame $F$ of $N_0$ along the curve is free. Changing it by $\exp(\varphi e_3)$ changes the potentials by a gauge but not the solution, which is unique. The harmonic map $\nu$ is an immersion near the diagonal exactly when $N_1$ and $N_0'-N_1$ are linearly independent, by~\eqref{eq:nu-derivs}.

We can now solve the geometric Cauchy problem for timelike CGC surfaces: to find a surface of prescribed constant curvature containing a prescribed curve with a prescribed normal along it. Combining Theorem~\ref{thm:cauchy} with Proposition~\ref{prop:cgc-sym}, the solution is read off algebraically from an extended frame whose potentials are written in the data.

\begin{corollary}
\label{cor:geometric-cauchy}
Let $I\subset\mathbb R$ be an open interval, and let $\tilde f:I\to\Hyp$ and $\tilde\nu:I\to\Sph$ satisfy $\langle\tilde f^{-1}\tilde f',\tilde\nu\rangle=0$. Let $\rho\in\mathbb R\setminus\{0,\pm1\}$, and put
\[
w=\tfrac12\big[\tilde\nu,\tilde f^{-1}\tilde f'\big],\qquad
N_1=\frac{\rho+1}{2\rho}\Big(\tilde\nu'+(\rho-1)w\Big),\qquad
\mu=\frac{\rho-1}{\rho+1}.
\]
Assume that $\tilde\nu'$ and $w$ are linearly independent at every point of $I$. Let $\hat F$ be the extended frame of Theorem~\ref{thm:cauchy} for the data $N_0=\tilde\nu$ and $N_1$. Then, on a neighbourhood $\Omega$ of the diagonal,
\[
f=\tilde f(x_0)\,\hat F|_{\lambda=\mu}\,\hat F|_{\lambda=1}^{-1}
\]
is a timelike immersion of constant extrinsic curvature $K+1=\rho^2$, with Gauss map $\nu=\Ad_{\hat F|_{\lambda=1}}e_3$ and asymptotic coordinates $(x,y)$, and
\[
f(x,x)=\tilde f(x),\qquad \nu(x,x)=\tilde\nu(x)\qquad (x\in I).
\]
\end{corollary}
\begin{proof}
We first identify the Cauchy data for $\nu$. Suppose that such an $f$ exists and is one of the surfaces of Proposition~\ref{prop:cgc-converse}. It then satisfies~\eqref{eq:cgc-ansatz} with $r=1/(1+\rho)$ and $s=1/(1-\rho)$, so along the diagonal
\[
f^{-1}f_t=f^{-1}f_x+f^{-1}f_y=\tfrac12\big[\nu,\,r\nu_x+s\nu_y\big].
\]
By Proposition~\ref{prop:bracket-identities}(v), $[\nu,[\nu,W]]=4W$ for every $W$ orthogonal to $\nu$; hence
\[
w=\tfrac12[\tilde\nu,\tilde f^{-1}\tilde f_t]=\tfrac14\big[\nu,[\nu,r\nu_x+s\nu_y]\big]=r\nu_x+s\nu_y .
\]
Combining this with $\nu'=\nu_x+\nu_y$ and solving the resulting linear system, using $r-s=-2\rho/(1-\rho^2)$,
\begin{equation}
\label{eq:nux-from-data}
\nu_x=\frac{w-s\,\nu'}{r-s}=\frac{\rho+1}{2\rho}\Big(\tilde\nu'+(\rho-1)w\Big)=N_1 .
\end{equation}
Now run the construction. We have $\langle N_1,\tilde\nu\rangle=0$, since $\langle\tilde\nu',\tilde\nu\rangle=0$ and $w\perp\tilde\nu$. So Theorem~\ref{thm:cauchy} provides a harmonic map $\nu$ with $\nu(x,x)=\tilde\nu(x)$ and $\nu_x(x,x)=N_1(x)$, and its extended frame $\hat F$. Along the diagonal $\nu_x$ and $\nu_y=\tilde\nu'-\nu_x$ are the two combinations $\frac{\rho+1}{2\rho}(\tilde\nu'+(\rho-1)w)$ and $\frac{\rho-1}{2\rho}(\tilde\nu'-(\rho+1)w)$ of $\tilde\nu'$ and $w$. Their determinant in that basis is $(1-\rho^2)/(2\rho)\ne0$, so they are independent exactly when $\tilde\nu'$ and $w$ are, and $\nu$ is an immersion on $\Omega$ after shrinking it. Proposition~\ref{prop:cgc-sym}, with $(1+\mu)/(1-\mu)=\rho$, then shows that $\hat F|_{\lambda=\mu}\hat F|_{\lambda=1}^{-1}$ is a timelike immersion of constant extrinsic curvature $\rho^2$ with Gauss map $\nu$ and asymptotic coordinates $(x,y)$. Its Maurer--Cartan form is~\eqref{eq:cgc-alpha}. Left translation by $\tilde f(x_0)$ changes neither of these properties.

Finally, $f$ contains the prescribed curve. Restricted to the diagonal, with $x=y=t$, the one-form~\eqref{eq:cgc-alpha} is $\tfrac12[\nu,r\nu_x+s\nu_y]\,dt$, and by construction $r\nu_x+s\nu_y=w$ there. Indeed, substituting $\nu_y=\nu'-\nu_x$ and~\eqref{eq:nux-from-data},
\[
r\nu_x+s\nu_y=(r-s)\nu_x+s\tilde\nu'=-\frac{\tilde\nu'+(\rho-1)w}{1-\rho}+\frac{\tilde\nu'}{1-\rho}=w .
\]
Hence $f^{-1}f_t=\tfrac12[\nu,w]=\tfrac14[\tilde\nu,[\tilde\nu,\tilde f^{-1}\tilde f_t]]=\tilde f^{-1}\tilde f_t$ along the diagonal, again by Proposition~\ref{prop:bracket-identities}(v). Since $\hat F|_\lambda(x_0,x_0)=X(x_0)=F(x_0)$ for every $\lambda$, we also have $f(x_0,x_0)=\tilde f(x_0)$. So $f$ and $\tilde f$ solve the same equation along the diagonal with the same initial value, and they agree there.
\end{proof}

The geometric Cauchy problem is thus solved by the DPW method followed by an algebraic evaluation. The potentials are written in the data through~\eqref{eq:cauchy-coefficients} and $N_1$, and the one-form~\eqref{eq:cgc-alpha} need not be integrated. This is the counterpart for constant Gaussian curvature surfaces in $\Hyp$ of the solution in~\cite{BS13} for timelike constant mean curvature surfaces in $\mathbb R^{2,1}$. The only hypothesis beyond the compatibility $\tilde\nu\perp\tilde f^{-1}\tilde f'$ is that $\tilde\nu'$ and $w$ be independent, which is what makes the Gauss map an immersion. No causal condition on the data is needed.

\begin{example}[the construction carried out]
\label{ex:cauchy}
We run Theorem~\ref{thm:cauchy} and Corollary~\ref{cor:geometric-cauchy} on data for which every step is explicit. Let $I\subset(-\infty,0)$ be an open interval and take, on $I$,
\[
N_0(t)=-\frac{1}{\sinh 2t}\,e_1-\coth(2t)\,e_3,
\qquad
N_1(t)=\frac{\cosh 2t}{\sinh^2 2t}\,e_1-\coth(2t)\,e_2+\frac{1}{\sinh^2 2t}\,e_3 .
\]
Both take values in $\Sph$: using $\langle e_1,e_1\rangle=-1$ and $\langle e_2,e_2\rangle=\langle e_3,e_3\rangle=1$,
\[
\langle N_0,N_0\rangle=\frac{\cosh^2 2t-1}{\sinh^2 2t}=1,
\qquad
\langle N_1,N_1\rangle=\frac{1-\cosh^2 2t}{\sinh^4 2t}+\frac{\cosh^2 2t}{\sinh^2 2t}=1 ,
\]
and $\langle N_1,N_0\rangle=\cosh 2t/\sinh^3 2t-\cosh 2t/\sinh^3 2t=0$, so the data are admissible. Along $I$ take the frame
\[
F(t)=\frac{\Id-e^{2t}e_2}{\sqrt{1-e^{4t}}},
\]
for which $\Ad_Fe_3=N_0$. Writing $\varsigma=2t$ for brevity, one finds $F^{-1}F'=e_2/\sinh\varsigma$, with no $\mathfrak k$-part, and~\eqref{eq:cauchy-coefficients} gives
\begin{equation}
\label{eq:example-abc}
\beta_1=\frac{\coth\varsigma}2\,e_1+\frac{1}{2\sinh\varsigma}\,e_2,\qquad
\gamma_{-1}=-\frac{\coth\varsigma}2\,e_1+\frac{1}{2\sinh\varsigma}\,e_2 .
\end{equation}
The potentials of the theorem are $\mathcal X(x)=\hat\alpha^\lambda(x)\,dx$ and $\mathcal Y(y)=\hat\alpha^\lambda(y)\,dy$, with $\hat\alpha^\lambda=\lambda\beta_1+\lambda^{-1}\gamma_{-1}$ evaluated at $\varsigma=2x$, respectively $\varsigma=2y$.

The solution produced is
\begin{equation}
\label{eq:example-nu}
\nu(x,y)=-\frac{1}{\sinh(x+y)}\,e_1-\coth(x+y)\Big(\sin(x-y)\,e_2+\cos(x-y)\,e_3\Big),
\end{equation}
on $I\times I$, where $x+y<0$. This is verified directly rather than by running the factorization, which is legitimate because the solution of the Cauchy problem is unique: writing $S=\sinh(x+y)$, $C=\cosh(x+y)$ and $u=\sin(x-y)e_2+\cos(x-y)e_3$, $u'=\cos(x-y)e_2-\sin(x-y)e_3$, so that $\nu=-S^{-1}e_1-CS^{-1}u$, one has
\[
\nu_x=\frac{C}{S^2}e_1+\frac1{S^2}u-\frac CS u',\qquad
\nu_y=\frac{C}{S^2}e_1+\frac1{S^2}u+\frac CS u' ,
\]
whence $\langle\nu,\nu\rangle=1$ and $\langle\nu_x,\nu_x\rangle=\langle\nu_y,\nu_y\rangle=1$. Differentiating once more, and using $u_y=-u'$, $u'_y=u$ and $C^2-S^2=1$,
\[
\nu_{xy}=-\frac{1+C^2}{S^3}\big(e_1+C\,u\big)=\frac{1+C^2}{S^2}\,\nu ,
\]
a multiple of $\nu$, so that $[\nu,\nu_{xy}]=0$ and $\nu$ is harmonic. Setting $x=y=t$ recovers $\nu(t,t)=N_0(t)$ and $\nu_x(t,t)=N_1(t)$.
The map $\nu$ is an immersion: $\{e_1,u,u'\}$ is a basis of $\sltwo$, and $\nu_y=k\nu_x$ would force $k=1$ on comparing the $e_1$-components, since $C\ne0$, and then $C/S=-C/S$ on comparing the $u'$-components, which is impossible.

This one harmonic map now yields surfaces. For each $0<\rho\ne1$, Proposition~\ref{prop:cgc-converse} integrates~\eqref{eq:cgc-alpha} to a timelike immersion of constant extrinsic curvature $\rho^2$ with Gauss map~\eqref{eq:example-nu}, and Corollary~\ref{cor:geometric-cauchy} identifies the curve it contains: the one obtained by integrating $\alpha$ along the diagonal, along which $w=r\nu_x+s\nu_y$ with $r=1/(1+\rho)$ and $s=1/(1-\rho)$. The one-form to be integrated is explicit:
\[
[\nu,\nu_x]=-\frac2S\,u'+\frac{2C}{S^2}\big(u+C\,e_1\big),
\qquad
[\nu,\nu_y]=-\frac2S\,u'-\frac{2C}{S^2}\big(u+C\,e_1\big),
\]
as follows from $[e_1,u]=-2u'$, $[e_1,u']=2u$, $[u,u']=2e_1$ and $1-C^2=-S^2$. Both branches of Proposition~\ref{prop:cgc-converse} are reached: $\rho>1$ gives $K>0$ and $0<\rho<1$ gives $-1<K<0$, from this one harmonic map. The mean curvature is explicit as well. Here $\langle\nu_x,\nu_y\rangle=-(1+C^2)/S^2$ and $[\nu_x,\nu_y]=-(4C/S^2)\nu$, so $\kappa=-4C/S^2$, and Lemma~\ref{lem:cgc-curvature} with $(1-r)/r=\rho$ gives
\[
H=-\rho\,\frac{1+\cosh^2(x+y)}{2\cosh(x+y)},
\]
which is not constant. For $0<|\rho|<1$ these are therefore surfaces with $-1<K<0$ and nonconstant mean curvature, and by Proposition~\ref{prop:no-parallel-cmc} none of their parallel surfaces has constant mean curvature. The immersion itself can also be written in closed form. The map $\nu$ is equivariant,
\[
\nu(x,y)=\Ad_{\exp(-\frac12(x-y)e_1)}\,\nu\big(\tfrac{x+y}2,\tfrac{x+y}2\big),
\]
so~\eqref{eq:cgc-alpha} satisfies $\alpha(x+c,y-c)=\Ad_{\exp(-ce_1)}\alpha(x,y)$, and the surface has the form $f=\exp\big((x-y)X\big)\,f_0(x+y)\,\exp\big(\tfrac12(x-y)e_1\big)$ with $X\in\sltwo$ constant and $f_0$ the solution of a linear ordinary differential equation in $x+y$. That equation can be solved explicitly. Put $Z=e^{x+y}\in(0,1)$ and $B_\rho=\rho/(\rho^2-1)$, and let $\gamma\in\mathbb R$ be defined by $\cos\gamma=(\rho^2-1)/(\rho^2+1)$ and $\sin\gamma=2\rho/(\rho^2+1)$. With $a_\pm=B_\rho(x+y)\pm\tfrac\gamma2$ put
\[
W=\frac{1}{1-Z^2}
\begin{pmatrix}-\cos a_+-Z\sin a_-&\sin a_++Z\cos a_-\\-\sin a_++Z\cos a_-&-\cos a_++Z\sin a_-\end{pmatrix}
\begin{pmatrix}1&Z\\Z&1\end{pmatrix}.
\]
Then, up to left translation,
\[
f(x,y)=\exp\big(c_\rho(x-y)e_1\big)\,W\,\exp\big(\tfrac12(x-y)e_1\big),
\]
where $c_\rho=(1+\rho^2)/\big(2(1-\rho^2)\big)$ and $\exp(ce_1)=\cos c\,\Id+\sin c\,e_1$. A direct computation shows that $\det f=1$ and that $f^{-1}df$ is the one-form~\eqref{eq:cgc-alpha}; by the equivariance it suffices to carry it out on the diagonal. The oscillating factors $\cos a_\pm$, $\sin a_\pm$ reflect the complex exponents $\pm iB_\rho$ of that equation at $x+y=-\infty$.
\end{example}

\begin{example}[a harmonic map taken from the literature]
\label{ex:akamine-kiyohara}
The input of Proposition~\ref{prop:cgc-converse} is a harmonic immersion into $\Sph$, and such maps are produced by methods unrelated to the ones used here. Akamine and Kiyohara~\cite{AK26} construct Lorentzian harmonic maps into $\Sph$ from framed null curves; we run one of them through the construction of this section, both because it supplies surfaces and because of what it shows about the hypotheses of this subsection.

Their Example~6.3 --- in the notation of~\cite[Thm.~1.2]{AK26}, the data $h(\sigma)=\tanh\sigma$ and the value $1$ of their constant $H$, which is not a mean curvature here --- is the map
\[
\nu(\sigma,\tau)=\Big(-\sinh 2\sigma-\tfrac\tau2\cosh2\sigma\Big)e_1+\tfrac\tau2\,e_2+\Big(\cosh2\sigma+\tfrac\tau2\sinh2\sigma\Big)e_3 ,
\]
harmonic for the Lorentz metric $\tau^2d\sigma^2-2\,d\sigma\,d\tau$ on the domain. That metric factors as $d\sigma\,(\tau^2d\sigma-2d\tau)$, so its null curves are $\{\sigma=\mathrm{const}\}$ and the solutions of $d\tau/d\sigma=\tau^2/2$, that is $\{-2/\tau-\sigma=\mathrm{const}\}$. Null coordinates $(x,y)$ are therefore given by $\sigma=y$ and $\tau=-2/(x+y)$, and in them
\begin{equation}
\label{eq:ak-nu}
\nu(x,y)=\Big(\!-\!\sinh 2y+\frac{\cosh 2y}{x+y}\Big)e_1-\frac{1}{x+y}\,e_2+\Big(\cosh 2y-\frac{\sinh 2y}{x+y}\Big)e_3
\end{equation}
on $\{x+y>0\}$. One checks directly that $\langle\nu,\nu\rangle=1$, using $\cosh^22y-\sinh^22y=1$, and that $\nu$ is harmonic. Writing $\nu_\tau=\tfrac12(-\cosh2y\,e_1+e_2+\sinh2y\,e_3)$ for the $\tau$-derivative, one finds $\langle\nu_\tau,\nu_\tau\rangle=0$, and since $\sigma=y$ does not involve $x$ we have $\nu_x=\frac{2}{(x+y)^2}\nu_\tau$; a short computation in the coordinates $(\sigma,\tau)$, where $\langle\nu_\sigma,\nu_\sigma\rangle=\tau^2-4$ and $\langle\nu_\sigma,\nu_\tau\rangle=-1$, then gives
\begin{equation}
\label{eq:ak-gram}
\langle\nu_x,\nu_x\rangle=0,\qquad \langle\nu_y,\nu_y\rangle=-4,\qquad \langle\nu_x,\nu_y\rangle=-\frac{2}{(x+y)^2} .
\end{equation}
The Gram determinant is $-4/(x+y)^4\ne0$, so $\nu_x$ and $\nu_y$ are linearly independent and $\nu$ is an immersion. Proposition~\ref{prop:cgc-converse} therefore applies: for every $\rho\ne0,\pm1$, integrating~\eqref{eq:cgc-alpha} produces a timelike immersion of $\{x+y>0\}$ into $\Hyp$ with constant extrinsic curvature $K+1=\rho^2$ and Gauss map~\eqref{eq:ak-nu}, with $(x,y)$ asymptotic.

Here everything is explicit. In the coordinates $(\sigma,\tau)$ one has $\nu=\Ad_Pe_3$ with $P=\exp(\sigma e_2)\exp(\tfrac\tau2\xi_2)$, and in the null coordinates $P^{-1}P_x=\tfrac{\tau^2}4\xi_2$ and $P^{-1}P_y=e_2+\tfrac\tau2e_3$. So, in the notation~\eqref{eq:frame-split}, $\beta_0=0$, $\beta_1=\tfrac{\tau^2}4\xi_2$, $\gamma_0=\tfrac\tau2e_3$ and $\gamma_{-1}=e_2$. One checks directly that
\[
\hat F=\exp\big(\lambda^{-1}y\,e_2\big)\exp\big(\tfrac{\lambda\tau}2\,\xi_2\big),\qquad \tau=-\frac2{x+y},
\]
satisfies $\hat F^{-1}d\hat F=\hat\alpha^\lambda$, so it is an extended frame of $\nu$. Proposition~\ref{prop:cgc-sym} then gives the surfaces in closed form: for $\mu\ne0,\pm1$,
\[
g^\mu=\exp\big(\mu^{-1}y\,e_2\big)\Big(\Id+\tfrac{(\mu-1)\tau}2\,\xi_2\Big)\exp(-y\,e_2),
\]
with $\exp(ce_2)=\cosh c\,\Id+\sinh c\,e_2$. In terms of $\rho=(1+\mu)/(1-\mu)$,
\[
g^\mu(x,y)=\exp\Big(-\frac{1+\rho}{1-\rho}\,y\,e_2\Big)\Big(\Id+\frac{2}{(1+\rho)(x+y)}\,\xi_2\Big)\exp(-y\,e_2).
\]
These surfaces are very special, however. Their mean curvature is the constant $H=\rho$, so $H^2=K+1$ and the shape operator has a double eigenvalue. Since $\langle f_x,f_x\rangle=0\ne\langle f_y,f_y\rangle$, the shape operator is a nontrivial Jordan block (at an umbilic point $II$ would be proportional to $I$, which with $L=\mathcal N=0$ would force $\langle f_y,f_y\rangle=0$). So they are isoparametric: they are the surfaces with nondiagonalizable shape operator mentioned in the introduction.

The second point concerns the causal character of the derivatives. By~\eqref{eq:ak-gram} neither null direction is spacelike: the $x$-direction is null and the $y$-direction timelike. No change of null coordinates alters this, since such a change multiplies $\langle\nu_x,\nu_x\rangle$ and $\langle\nu_y,\nu_y\rangle$ by positive factors. The extended frame above shows that this is no obstruction to the loop group description: Lemma~\ref{lem:cgc-mc}, Proposition~\ref{prop:cgc-sym} and Theorem~\ref{thm:cauchy} require no causal condition on $\nu_x$ or $\nu_y$.
\end{example}

\begin{remark}
\label{rem:integration-cost}
In both constructions the surface is recovered from its own extended frame by the same algebraic evaluation: Proposition~\ref{prop:sym-bobenko} on the constant mean curvature side and Proposition~\ref{prop:cgc-sym} on the constant Gaussian curvature side. As written, the two frames look different: a constant mean curvature frame is built in null coordinates for the first fundamental form, with the spectral parameter attached to $A^1\xi_1$ and $B^2\xi_2$ only, while a constant Gaussian curvature frame is $\sigma$-twisted. The next section shows that the difference is only one of normalization: after the substitution $\lambda=\zeta^2$ and a constant gauge, a constant mean curvature extended frame is an extended frame of its Gauss map in the sense of \S\ref{sec:cgc-cauchy}, and both classes of surfaces, in both curvature branches, are read off from it.
\end{remark}

\section{The relation between the two classes}
\label{sec:bridge}

Sections~\ref{sec:cmc} and~\ref{sec:cgc} produced two loop group constructions which, as written, look unrelated: they are adapted to different conformal structures on $\Omega$, their frames are normalized differently, and each recovers only surfaces of its own class. We first recall the geometric relation between the two classes, parallel displacement (\S\ref{sec:cgc-parallel}), and write it on the extended frame (\S\ref{sec:bridge-identity}). We then show that a constant mean curvature extended frame is, up to normalization, an extended frame of its Gauss map in the sense of \S\ref{sec:cgc-cauchy} (Lemma~\ref{lem:cmc-twisted}). This yields constant Gaussian curvature surfaces of both curvature branches from a single constant mean curvature frame (Theorem~\ref{thm:bridge}), and it solves the geometric Cauchy problem for constant mean curvature surfaces (\S\ref{sec:cauchy-cmc}).

\subsection{Parallel surfaces}
\label{sec:cgc-parallel}

The classical relation between constant mean curvature and constant Gaussian curvature surfaces --- in Euclidean space, Bonnet's theorem: every surface of nonzero constant mean curvature is parallel to one of constant positive curvature --- is known in the Riemannian space forms $M^3(\bar K)$ of curvature $\bar K=0,\pm1$, together with the curvatures of the whole parallel family, in~\cite[Chap.~I, Prop.~3.2 and Cor.~3.4]{Tenenblat98}. For timelike surfaces in a Lorentzian space form $M^3_1(c)$, $c=0,\pm1$, the curvatures of the parallel surfaces are computed in~\cite[(3.13)]{Zuo06}, where it is shown that the surfaces parallel to a timelike surface of constant Gaussian curvature are linear Weingarten, and that one of them has constant mean curvature~\cite[Thm.~3.4 and Rem.~3.5]{Zuo06}; the flat case goes back to~\cite{Ino98}. Lemma~\ref{lem:parallel-curvatures} and Theorem~\ref{thm:parallel} below are therefore known for $\Hyp$, and no novelty is claimed for them; the statements are formally the same as in the Riemannian case of curvature $-1$, as Remark~\ref{rem:tenenblat} records. They are the geometric fact underlying the rest of this section, and we prove them, in the present notation, in the form in which they are used below. We first record an elementary matrix identity.

\begin{proposition}
\label{prop:matrix-identity}
Let $A$ be a $2\times2$ matrix, let $\eta,\zeta\in\mathbb R$ and set $B=\eta A+\zeta\Id$. Then
\[
\det B=\eta^2\det A+\eta\zeta\operatorname{tr}A+\zeta^2,\qquad \operatorname{tr}B=\eta\operatorname{tr}A+2\zeta,
\]
and, if $\det B\ne0$,
\[
B^{-1}=\frac{\eta\bar A+\zeta\Id}{\det B},
\]
where $\bar A=(\operatorname{tr}A)\Id-A$ is the cofactor transpose.
\end{proposition}

Given a timelike immersion $f:\Omega\to\Hyp\subset\mathbb R^{2,2}$ with unit normal $N$ and $\theta\in\mathbb R$, define the \emph{parallel surface at distance $\theta$} by
\begin{equation}
\label{eq:parallel}
f^\theta=f\cosh\theta+N\sinh\theta .
\end{equation}
It takes values in $\Hyp$, since $\langle f^\theta,f^\theta\rangle=-\cosh^2\theta+\sinh^2\theta=-1$, and geometrically it moves each point a distance $\theta$ along the geodesic through it in the direction of $N$. As in the Euclidean case it need not be an immersion. Differentiating and using the shape operator,
\begin{equation}
\label{eq:parallel-derivatives}
\begin{pmatrix}f^\theta_x\\ f^\theta_y\end{pmatrix}=(\Id\cosh\theta+S\sinh\theta)\begin{pmatrix}f_x\\ f_y\end{pmatrix},
\end{equation}
so $f^\theta$ is an immersion exactly when $\Id\cosh\theta+S\sinh\theta$ is nonsingular, which by Proposition~\ref{prop:matrix-identity} and~\eqref{eq:gauss-general} means
\begin{equation}
\label{eq:parallel-nondeg}
\Delta_\theta:=\det(\Id\cosh\theta+S\sinh\theta)=K\sinh^2\theta-H\sinh(2\theta)+\cosh(2\theta)\ne0 .
\end{equation}
Assume this from now on. Then
\[
N^\theta=N\cosh\theta+f\sinh\theta
\]
is a unit normal for $f^\theta$. Unlike in the Euclidean case the normal itself changes, but the Gauss map does not: since $\nu$ lies in $\sltwo$ we have $\bar\nu=-\nu$ and $\det\nu=-\langle\nu,\nu\rangle=-1$, so $\nu^{-1}=\bar\nu/\det\nu=\nu$, and therefore
\begin{equation}
\label{eq:gauss-invariant}
\nu^\theta=(f^\theta)^{-1}N^\theta=(f^{-1}\cosh\theta-N^{-1}\sinh\theta)(f\sinh\theta+N\cosh\theta)=\nu\cosh^2\theta-\nu^{-1}\sinh^2\theta=\nu .
\end{equation}
This is the key point: a whole family of surfaces, of varying curvature, shares one Gauss map.

\begin{lemma}
\label{lem:parallel-curvatures}
Assume~\eqref{eq:parallel-nondeg}. The shape operator of $f^\theta$ is
\[
S^\theta=(S\cosh\theta+\Id\sinh\theta)(\Id\cosh\theta+S\sinh\theta)^{-1}
=\frac{S+\big[\tfrac12(K+2)\sinh 2\theta-2H\sinh^2\theta\big]\Id}{\Delta_\theta},
\]
and consequently
\[
K^\theta=\frac{K}{\Delta_\theta},\qquad
H^\theta=\frac{H\cosh(2\theta)-\tfrac12(K+2)\sinh(2\theta)}{\Delta_\theta} .
\]
\end{lemma}
\begin{proof}
Differentiating $N^\theta$ gives, exactly as in~\eqref{eq:parallel-derivatives},
\[
\begin{pmatrix}N^\theta_x\\ N^\theta_y\end{pmatrix}=(\Id\sinh\theta+S\cosh\theta)\begin{pmatrix}f_x\\ f_y\end{pmatrix},
\]
and substituting this and~\eqref{eq:parallel-derivatives} into the defining relation of $S^\theta$ yields the matrix identity $S^\theta(\Id\cosh\theta+S\sinh\theta)=\Id\sinh\theta+S\cosh\theta$, which is the first displayed formula. For the second, apply Proposition~\ref{prop:matrix-identity} with $A=S$, $\eta=\sinh\theta$, $\zeta=\cosh\theta$, so that
\[
(\Id\cosh\theta+S\sinh\theta)^{-1}=\frac{(\sinh\theta\operatorname{tr}S+\cosh\theta)\Id-\sinh\theta\,S}{\Delta_\theta},
\]
and expand the numerator of $S^\theta$, using the Cayley--Hamilton identity $S^2=(\operatorname{tr}S)S-(\det S)\Id$:
the terms in $S$ collect to $(\cosh^2\theta-\sinh^2\theta)S=S$, while the terms in $\Id$ collect to
\[
\big[\sinh\theta\cosh\theta(\det S+1)+\sinh^2\theta\operatorname{tr}S\big]\Id=\big[\tfrac12(K+2)\sinh2\theta-2H\sinh^2\theta\big]\Id,
\]
using $\det S=K+1$ and $\operatorname{tr}S=-2H$. Taking traces gives
\[
H^\theta=-\tfrac12\operatorname{tr}S^\theta=\frac{H(1+2\sinh^2\theta)-\tfrac12(K+2)\sinh2\theta}{\Delta_\theta}
=\frac{H\cosh2\theta-\tfrac12(K+2)\sinh 2\theta}{\Delta_\theta}.
\]
Taking determinants instead in $S^\theta(\Id\cosh\theta+S\sinh\theta)=\Id\sinh\theta+S\cosh\theta$ and using Proposition~\ref{prop:matrix-identity} twice,
\[
K^\theta+1=\det S^\theta=\frac{K\cosh^2\theta+\cosh2\theta-H\sinh2\theta}{\Delta_\theta},
\]
so that
\[
K^\theta=\frac{K\cosh^2\theta-K\sinh^2\theta}{\Delta_\theta}=\frac{K}{\Delta_\theta}. \qedhere
\]
\end{proof}

\begin{theorem}
\label{thm:parallel}
Let $f:\Omega\to\Hyp$ be a timelike immersion with mean curvature $H$ and Gaussian curvature $K$, and let $\theta\ne0$.
\begin{enumerate}\itemsep=0pt
\item[\rm(1)] If $f$ has constant mean curvature $H$ with $|H|>1$ --- which is exactly the condition for a real $\theta\ne0$ with $\tanh(2\theta)=1/H$ to exist, since $|\tanh|<1$ --- and $\theta$ is that value, then $\Delta_\theta=K\sinh^2\theta$, so $f^\theta$ is an immersion exactly at the points where $K\ne0$, and there it has constant Gaussian curvature
\[
K^\theta=\frac1{\sinh^2\theta}>0,\qquad\text{equivalently}\qquad K^\theta+1=\coth^2\theta .
\]
\item[\rm(2)] If $f$ has constant Gaussian curvature $K>0$ and $\theta$ satisfies $\sinh^2\theta=1/K$, then $\Delta_\theta=2\cosh\theta(\cosh\theta-H\sinh\theta)$, so $f^\theta$ is an immersion exactly where $H\ne\coth\theta$, and there it has constant mean curvature
\[
H^\theta=-\coth(2\theta).
\]
\end{enumerate}
In both cases $f^\theta$ has the same Gauss map as $f$.
\end{theorem}
\begin{proof}
(1) If $\tanh(2\theta)=1/H$ then $-H\sinh(2\theta)+\cosh(2\theta)=\cosh(2\theta)(1-H\tanh 2\theta)=0$, so~\eqref{eq:parallel-nondeg} reduces to $\Delta_\theta=K\sinh^2\theta$, which since $\theta\ne0$ is nonzero exactly where $K\ne0$. There, Lemma~\ref{lem:parallel-curvatures} gives $K^\theta=K/(K\sinh^2\theta)=1/\sinh^2\theta$, a constant, and $K^\theta+1=1+1/\sinh^2\theta=\coth^2\theta$.

(2) If $\sinh^2\theta=1/K$ then $K\sinh^2\theta=1$ and
\[
\Delta_\theta=1+\cosh2\theta-H\sinh2\theta=2\cosh^2\theta-2H\sinh\theta\cosh\theta=2\cosh\theta(\cosh\theta-H\sinh\theta),
\]
which vanishes exactly when $H=\coth\theta$. Elsewhere, using $K+2=(1+2\sinh^2\theta)/\sinh^2\theta=\cosh2\theta/\sinh^2\theta$,
\[
\tfrac12(K+2)\sinh2\theta=\frac{\cosh2\theta\cdot\sinh\theta\cosh\theta}{\sinh^2\theta}=\cosh2\theta\coth\theta,
\]
so by Lemma~\ref{lem:parallel-curvatures},
\[
H^\theta=\frac{\cosh2\theta\,(H-\coth\theta)}{2\cosh\theta(\cosh\theta-H\sinh\theta)}
=\frac{-\cosh2\theta\,(\cosh\theta-H\sinh\theta)}{2\sinh\theta\cosh\theta\,(\cosh\theta-H\sinh\theta)}=-\coth(2\theta),
\]
a constant. The statement about the Gauss map is~\eqref{eq:gauss-invariant}.
\end{proof}

Part~(1) of Theorem~\ref{thm:parallel} only produces positive curvatures, $K^\theta=1/\sinh^2\theta$. The next proposition shows that this is a feature of the correspondence and not of the proof: apart from isoparametric surfaces, a constant Gaussian curvature surface with $K\le0$ has no parallel surface of constant mean curvature, and one with $K>0$ has none other than those of part~(2).

\begin{proposition}
\label{prop:no-parallel-cmc}
Let $f:\Omega\to\Hyp$ be a timelike immersion of constant Gaussian curvature $K$ whose mean curvature is not constant on any nonempty open set, and let $\theta\ne0$ with $K\sinh^2\theta\ne1$. Then $f^\theta$ does not have constant mean curvature on any nonempty open set on which it is an immersion. In particular, if $K\le0$ no parallel surface of $f$ has constant mean curvature on an open set.
\end{proposition}
\begin{proof}
By Lemma~\ref{lem:parallel-curvatures}, wherever $\Delta_\theta\ne0$,
\[
H^\theta=\frac{\alpha H+\beta}{\gamma H+\delta},
\]
where
\[
\alpha=\cosh2\theta,\quad \beta=-\tfrac12(K+2)\sinh2\theta,\quad \gamma=-\sinh2\theta,\quad \delta=K\sinh^2\theta+\cosh2\theta,
\]
with $\gamma H+\delta=\Delta_\theta$ by~\eqref{eq:parallel-nondeg}. Since $K$ is constant, so are $\alpha,\beta,\gamma,\delta$, and, using $\cosh^22\theta-\sinh^22\theta=1$ and $\cosh2\theta-2\cosh^2\theta=-1$,
\[
\alpha\delta-\beta\gamma=K\sinh^2\theta\cosh2\theta+\cosh^22\theta-\tfrac12(K+2)\sinh^22\theta=1-K\sinh^2\theta\ne0 .
\]
So $H\mapsto H^\theta$ is an injective fractional linear map, and if $H^\theta$ is constant on a connected open set then so is $H$, contrary to the hypothesis.
\end{proof}

Surfaces to which Proposition~\ref{prop:no-parallel-cmc} applies with $-1<K<0$ are given in Example~\ref{ex:cauchy}. Since a CMC surface and its parallel CGC surface are parallel to each other, the branch $-1<K<0$ cannot be reached from CMC surfaces by parallel displacement. Theorem~\ref{thm:bridge} below reaches it by a different operation, the passage to the polar surface (Remark~\ref{rem:bridge-parallel}).

\begin{remark}
\label{rem:tenenblat}
Theorem~\ref{thm:parallel} is formally identical to the classical statement in the Riemannian space form of curvature $-1$~\cite[Chap.~I, Prop.~3.2 and Cor.~3.4]{Tenenblat98}: setting $\bar K=-1$ and $c=H$ there returns the distances, hypotheses and curvatures of both parts; for instance, with $H=\coth2\theta$ the curvature $2(H^2-1)+2H\sqrt{H^2-1}$ given there equals $1/\sinh^2\theta$. The surfaces there are Riemannian, so Theorem~\ref{thm:parallel} does not follow from that statement; the timelike statement is the one of~\cite{Zuo06} quoted above, and the agreement is an independent check on Lemma~\ref{lem:parallel-curvatures}.
\end{remark}

\subsection{The identity}
\label{sec:bridge-identity}

\begin{lemma}[parallel surfaces as evaluations of the frame]
\label{lem:bridge-identity}
Let $\hat F:\Omega\to G$ be a regular admissible extended frame, $F=\hat F|_{\lambda=1}$, and for $\mu\ne0,1$ let $f^\mu$ and $N^\mu$ be as in Proposition~\ref{prop:sym-bobenko}. Then for every $\theta\in\mathbb R$,
\begin{equation}
\label{eq:bridge}
f^\mu\cosh\theta+N^\mu\sinh\theta=\hat F|_{\lambda=\mu}\;\exp(\theta e_3)\;\hat F|_{\lambda=1}^{-1} .
\end{equation}
\end{lemma}
\begin{proof}
Write $\mathcal G=\hat F|_{\lambda=\mu}$. By Proposition~\ref{prop:sym-bobenko}, $f^\mu=\mathcal GF^{-1}$ and $N^\mu=f^\mu\Ad_Fe_3=\mathcal Ge_3F^{-1}$, so
\[
f^\mu\cosh\theta+N^\mu\sinh\theta=\mathcal G\big(\Id\cosh\theta+e_3\sinh\theta\big)F^{-1} .
\]
Since $e_3^2=\Id$, the exponential series splits into even and odd parts and gives $\exp(\theta e_3)=\Id\cosh\theta+e_3\sinh\theta$, which is the claim.
\end{proof}

Since $\theta$ is so far unconstrained, Lemma~\ref{lem:bridge-identity} exhibits a whole one-parameter family, and it is worth computing it once and for all: everything in this section is a specialization of the following lemma to a particular value of the parameter.

\begin{lemma}[the interpolating family]
\label{lem:interpolation}
Let $\hat F$ be a regular admissible extended frame, $F=\hat F|_{\lambda=1}$, and $\mu\ne0,1$. For $t\in\mathbb R$ put
\[
g^t:=\hat F|_{\lambda=\mu}\,\exp(te_3)\,\hat F|_{\lambda=1}^{-1},
\]
so that $g^t=(f^\mu)^t$ by Lemma~\ref{lem:bridge-identity}. Then
\begin{align}
\label{eq:interp-x}
(g^t)^{-1}g^t_x&=\Ad_F\Big(A^1\big(\mu e^{-2t}-1\big)\xi_1+A^2\big(e^{2t}-1\big)\xi_2\Big),\\
\label{eq:interp-y}
(g^t)^{-1}g^t_y&=\Ad_F\Big(B^1\big(e^{-2t}-1\big)\xi_1+B^2\big(\mu^{-1}e^{2t}-1\big)\xi_2\Big),
\end{align}
and, writing $N^t=g^t\Ad_Fe_3$,
\begin{gather}
\label{eq:interp-EG}
\langle g^t_x,g^t_x\rangle=A^1A^2\big(\mu e^{-2t}-1\big)\big(e^{2t}-1\big),\qquad
\langle g^t_y,g^t_y\rangle=B^1B^2\big(e^{-2t}-1\big)\big(\mu^{-1}e^{2t}-1\big),\\
\label{eq:interp-LN}
\langle g^t_{xx},N^t\rangle=A^1A^2\big(e^{2t}-\mu e^{-2t}\big),\qquad
\langle g^t_{yy},N^t\rangle=B^1B^2\big(\mu^{-1}e^{2t}-e^{-2t}\big).
\end{gather}
Consequently, for $\mu>0$:
\begin{enumerate}\itemsep=0pt
\item[\rm(i)] if $A^1A^2B^1B^2\ne0$, the coordinates $(x,y)$ are null for $g^t$ exactly for $t=0$ and $t=\tfrac12\log\mu$;
\item[\rm(ii)] the coefficients $L$ and $\mathcal N$ of the second fundamental form of $g^t$ vanish at $t=\tfrac14\log\mu$, and, if $A^1A^2\ne0$ or $B^1B^2\ne0$, at no other value of $t$;
\item[\rm(iii)] at $t=\tfrac14\log\mu$, $g^t$ is an immersion if and only if $A^1B^2-A^2B^1\ne0$, and then $(x,y)$ are asymptotic coordinates for it.
\end{enumerate}
\end{lemma}
\begin{proof}
Write $E=\exp(te_3)$ and $\mathcal G=\hat F|_{\lambda=\mu}$, so $g^t=\mathcal GEF^{-1}$ and, exactly as in the proof of Proposition~\ref{prop:sym-bobenko},
\[
(g^t)^{-1}g^t_x=\Ad_F\big(\Ad_{E^{-1}}(\mathcal G^{-1}\mathcal G_x)-F^{-1}F_x\big).
\]
By~\eqref{eq:xi-brackets}, $\Ad_{\exp(-te_3)}\xi_1=e^{-2t}\xi_1$ and $\Ad_{\exp(-te_3)}\xi_2=e^{2t}\xi_2$, while $e_3$ is fixed; since $\mathcal G^{-1}\mathcal G_x=A^1\mu\,\xi_1+A^2\xi_2+A^3e_3$,
\[
\Ad_{E^{-1}}(\mathcal G^{-1}\mathcal G_x)=A^1\mu e^{-2t}\xi_1+A^2e^{2t}\xi_2+A^3e_3,
\]
and subtracting $F^{-1}F_x=A^1\xi_1+A^2\xi_2+A^3e_3$ gives~\eqref{eq:interp-x}; the $y$-side is identical, starting from $\mathcal G^{-1}\mathcal G_y=B^1\xi_1+B^2\mu^{-1}\xi_2+B^3e_3$.

Writing $(g^t)^{-1}g^t_x=\Ad_F(a\xi_1+b\xi_2)$ with $a,b$ as in~\eqref{eq:interp-x}, the first identity in~\eqref{eq:interp-EG} is $\langle a\xi_1+b\xi_2,a\xi_1+b\xi_2\rangle=2ab\langle\xi_1,\xi_2\rangle=ab$, and the second is the same computation. For~\eqref{eq:interp-LN}, put $u=(g^t)^{-1}g^t_x$; then $\langle g^t_{xx},N^t\rangle$ is the $e_3$-component of $\Ad_F^{-1}(u^2+u_x)$, and $u^2=\langle u,u\rangle\Id$ is orthogonal to $\sltwo$ by~\eqref{eq:naive-second-derivatives}, so only $u_x$ contributes. Its $e_3$-component comes from
\[
[F^{-1}F_x,a\xi_1+b\xi_2]=a\big(2A^3\xi_1-A^2e_3\big)+b\big(A^1e_3-2A^3\xi_2\big),
\]
and equals $bA^1-aA^2=A^1A^2\big[(e^{2t}-1)-(\mu e^{-2t}-1)\big]=A^1A^2(e^{2t}-\mu e^{-2t})$, which is the first identity in~\eqref{eq:interp-LN}; the second is the same computation on the $y$-side.

Finally, under the stated nonvanishing assumptions the right-hand sides of~\eqref{eq:interp-EG} vanish simultaneously exactly when $e^{2t}=1$ or $e^{2t}=\mu$; those of~\eqref{eq:interp-LN} both vanish when $e^{4t}=\mu$, and a nonzero factor $A^1A^2$ or $B^1B^2$ forces $e^{4t}=\mu$. At $e^{4t}=\mu$, writing $m=e^{2t}$ and $D=A^1B^2-A^2B^1$, the coefficients in~\eqref{eq:interp-x}--\eqref{eq:interp-y} are $A^1(m-1)$, $A^2(m-1)$, $B^1(m^{-1}-1)$ and $B^2(m^{-1}-1)$; the same computation as for~\eqref{eq:interp-LN}, applied to $M=\langle v_x+\tfrac12[u,v],N^t\rangle$ with $v=(g^t)^{-1}g^t_y$, gives $M=\tfrac12D(m^{-1}-1)(m+1)$, and $\det I=-\tfrac14(m-1)^2(m^{-1}-1)^2D^2$ follows from $pq-\tfrac14(p+q)^2=-\tfrac14(p-q)^2$ with $p=A^1B^2$, $q=A^2B^1$. Both vanish exactly when $D=0$.
\end{proof}

By the proof of Proposition~\ref{prop:degenerate-locus}, which uses only the frame, the condition $A^1B^2-A^2B^1\ne0$ in~(iii) says that the Gauss map $\nu=\Ad_Fe_3$ is an immersion. The flat family of \S\ref{sec:cmc-example}, where $A^1=A^2=B^1=B^2=-\tfrac12$, violates it everywhere.

The content of Lemma~\ref{lem:bridge-identity} is not the computation, which is immediate, but the observation that the parallel surface of $f^\mu$ --- an object defined through the geometry of $\Hyp$, involving the normal and the exponential map of the ambient space --- is again of Sym--Bobenko type, with the fixed element $\exp(\theta e_3)$ of the one-parameter subgroup generated by $e_3$ inserted between the two evaluations of the frame. Note also that $\theta$ and $\mu$ are so far independent. Tying them together is what produces constant curvature.

\subsection{The main theorem}

\begin{lemma}[the constant mean curvature frame as a twisted frame]
\label{lem:cmc-twisted}
Let $\hat F$ be an admissible extended frame as in~\eqref{eq:admissible}, $F=\hat F|_{\lambda=1}$ and $\nu=\Ad_Fe_3$. For $\zeta\in\mathbb R\setminus\{0\}$ put
\[
G(\zeta)=\hat F|_{\lambda=\zeta^2}\,\exp\big(\tfrac12\log\zeta\;e_3\big)\quad(\zeta>0),\qquad
G(\zeta)=e_3\,\hat F|_{\lambda=\zeta^2}\,e_3\,\exp\big(\tfrac12\log|\zeta|\;e_3\big)\quad(\zeta<0).
\]
Then $G(\zeta)$ takes values in $\SLR$, $G(1)=F$, and
\[
G^{-1}dG=\big(A^3e_3+\zeta(A^1\xi_1+A^2\xi_2)\big)\,dx+\big(B^3e_3+\zeta^{-1}(B^1\xi_1+B^2\xi_2)\big)\,dy .
\]
Thus, up to a left factor constant on $\Omega$, $\zeta\mapsto G(\zeta)$ is an extended frame of the harmonic map $\nu$ in the sense of \S\ref{sec:cgc-cauchy}, in the null coordinates $(x,y)$ of the first fundamental form, with $\beta_0=A^3e_3$, $\beta_1=A^1\xi_1+A^2\xi_2$, $\gamma_0=B^3e_3$ and $\gamma_{-1}=B^1\xi_1+B^2\xi_2$.
\end{lemma}
\begin{proof}
Conjugation by $e_3$, which has determinant $-1$, preserves $\SLR$, and a constant left factor does not change the Maurer--Cartan form. For a constant right factor $C$, the Maurer--Cartan form becomes $\Ad_{C^{-1}}$ of the old one. With $s=\tfrac12\log|\zeta|$, by~\eqref{eq:xi-brackets} the map $\Ad_{\exp(-se_3)}$ fixes $e_3$ and multiplies $\xi_1$ by $e^{-2s}=|\zeta|^{-1}$ and $\xi_2$ by $|\zeta|$, and $\Ad_{e_3}$ changes the signs of $\xi_1$ and $\xi_2$. Applied to~\eqref{eq:admissible} at $\lambda=\zeta^2$, the terms $\zeta^2A^1\xi_1$, $A^2\xi_2$, $B^1\xi_1$ and $\zeta^{-2}B^2\xi_2$ become $\zeta A^1\xi_1$, $\zeta A^2\xi_2$, $\zeta^{-1}B^1\xi_1$ and $\zeta^{-1}B^2\xi_2$, the sign of $\zeta$ coming from $\Ad_{e_3}$ when $\zeta<0$. The map $\nu$ is harmonic in $(x,y)$ by Theorem~\ref{thm:cmc-harmonic}, and $\beta_0,\gamma_0\in\mathfrak k$ and $\beta_1,\gamma_{-1}\in\mathfrak p$, which is the shape~\eqref{eq:cgc-extended}.
\end{proof}

\begin{theorem}
\label{thm:bridge}
Let $\hat F:\Omega\to G$ be a regular admissible extended frame, $F=\hat F|_{\lambda=1}$ and $\nu=\Ad_Fe_3$, and let $\theta\ne0$. At every point where $\nu$ is an immersion, the maps
\begin{equation}
\label{eq:bridge-formula}
g^\theta_+=\hat F|_{\lambda=e^{4\theta}}\;\exp(\theta e_3)\;\hat F|_{\lambda=1}^{-1},\qquad
g^\theta_-=e_3\,\hat F|_{\lambda=e^{4\theta}}\,e_3\;\exp(\theta e_3)\;\hat F|_{\lambda=1}^{-1}
\end{equation}
are timelike immersions into $\Hyp$ with Gauss map $\nu$ and with $(x,y)$ as asymptotic coordinates, of constant Gaussian curvature
\[
K_+=\frac1{\sinh^2\theta}>0\qquad\text{and}\qquad K_-=-\frac1{\cosh^2\theta}\in(-1,0)
\]
respectively. Every $K\in(-1,0)\cup(0,\infty)$ arises in this way.
\end{theorem}
\begin{proof}
By Lemma~\ref{lem:cmc-twisted}, $g^\theta_\pm=G(\zeta)G(1)^{-1}$ with $\zeta=\pm e^{2\theta}$, where $G$ is an extended frame of the harmonic immersion $\nu$. Since $\theta\ne0$ we have $\zeta\ne0,\pm1$, and Proposition~\ref{prop:cgc-sym} with $\mu=\zeta$ applies. It gives a timelike immersion with Gauss map $\nu$ and asymptotic coordinates $(x,y)$, with $K+1=\rho^2$ and $\rho=(1+\zeta)/(1-\zeta)$. For $\zeta=e^{2\theta}$ this is $\rho=-\coth\theta$, so $K=\coth^2\theta-1=1/\sinh^2\theta$. For $\zeta=-e^{2\theta}$ it is $\rho=-\tanh\theta$, so $K=\tanh^2\theta-1=-1/\cosh^2\theta$. As $\theta$ ranges over $\mathbb R\setminus\{0\}$, these take every value in $(0,\infty)$ and in $(-1,0)$ respectively.
\end{proof}

We write $g^\theta=g^\theta_+$ where no confusion can arise. The positive branch can also be obtained from Lemma~\ref{lem:interpolation} alone. At $t=\theta$, with $m=e^{2\theta}$ and $D=A^1B^2-A^2B^1$, the coefficients are $E=A^1A^2(m-1)^2$, $G=B^1B^2(m^{-1}-1)^2$, $L=\mathcal N=0$ and $M=\tfrac12D(m^{-1}-1)(m+1)$, with $\det I=-\tfrac14(m-1)^2(m^{-1}-1)^2D^2$. Hence $K+1=-M^2/\det I=\coth^2\theta$.

\begin{remark}[the geometric meaning of the formula]
\label{rem:bridge-parallel}
The formula has a classical meaning, which also explains the choice $\mu=e^{4\theta}$. By Lemma~\ref{lem:bridge-identity}, $g^\theta=(f^\mu)^\theta$ is the parallel surface at distance $\theta$ of the constant mean curvature immersion $f^\mu=\hat F|_{\lambda=\mu}F^{-1}$ of Proposition~\ref{prop:sym-bobenko}. The mean curvature of $f^\mu$ is
\[
H=\frac{\mu+1}{\mu-1}=\frac{e^{2\theta}+e^{-2\theta}}{e^{2\theta}-e^{-2\theta}}=\coth(2\theta),
\]
so the tie $\mu=e^{4\theta}$ is exactly the condition $\tanh(2\theta)=1/H$ of Theorem~\ref{thm:parallel}(1), with $|H|>1$. The Gauss map of $f^\mu$ is $\nu$: by Proposition~\ref{prop:sym-bobenko} its unit normal is $f^\mu\Ad_Fe_3$, and its adapted frame in the sense of Lemma~\ref{lem:adapted-frame} is $F\exp(\tau e_3)$ for a suitable $\tau$, which gives the same $\Ad_{F\exp(\tau e_3)}e_3=\nu$. Proposition~\ref{prop:degenerate-locus}, applied to $f^\mu$, therefore shows that $K^\mu\ne0$ exactly where $\nu$ is an immersion. Theorem~\ref{thm:parallel}(1) then gives a second proof of the curvature statement for $g^\theta_+$. In this reading, the positive branch of Theorem~\ref{thm:bridge} is the parallel-surface correspondence of Theorem~\ref{thm:parallel}, expressed on the extended frame.

The negative branch has a classical meaning as well. Since $e_3\exp(\theta e_3)=\sinh\theta\,\Id+\cosh\theta\,e_3$, the computation of Lemma~\ref{lem:bridge-identity} gives
\[
g^\theta_-=e_3\big(f^\mu\sinh\theta+N^\mu\cosh\theta\big)=e_3\,N^\theta ,
\]
where $N^\theta=N^\mu\cosh\theta+f^\mu\sinh\theta$ is the unit normal of the parallel surface $g^\theta_+=(f^\mu)^\theta$ of~\eqref{eq:parallel}. The normal $N^\theta$ takes values in the quadric $\{\langle X,X\rangle=1\}$ of $\mathbb R^{2,2}$. Left multiplication by $e_3$, which has determinant $-1$, carries that quadric onto $\Hyp$ and reverses the sign of the metric. So $g^\theta_-$ is the polar surface of $g^\theta_+$, carried into $\Hyp$ by this anti-isometry, and the extrinsic curvatures satisfy $(K_++1)(K_-+1)=\coth^2\theta\tanh^2\theta=1$, as for polar surfaces in $S^3$. By Proposition~\ref{prop:no-parallel-cmc}, $g^\theta_-$ is not a parallel surface of any constant mean curvature surface unless its mean curvature is constant.\end{remark}

\begin{remark}[the degenerate locus]
\label{rem:degenerate}
The hypothesis that $\nu$ be an immersion is not removable. Where it fails, $f^\mu$ is flat by Proposition~\ref{prop:degenerate-locus}, so the matrix $\Id\cosh\theta+S\sinh\theta$ is singular for the value of $\theta$ prescribed by $\tanh(2\theta)=1/H$, and $g^\theta$ fails to be an immersion there. The degenerate locus is therefore not an artifact of the construction: it is exactly the set where $\nu$ fails to be an immersion, which is the standing hypothesis of the whole of Section~\ref{sec:cgc} and was seen there to be a genuine restriction. In this sense the formula is defined exactly where the constant Gaussian curvature theory it lands in is itself defined. The locus depends only on $\nu$: not on $\mu$, not on $\theta$, and not on which member of the associated family one starts from. By Proposition~\ref{prop:constant-potential-flat} the flat example of \S\ref{sec:cmc-example} lies entirely on this locus, as does every example whose frame has constant coefficients; Example~\ref{ex:sinh-gordon} gives one that does not.
\end{remark}

\begin{remark}[an explicit family on which the theorem applies]
\label{rem:bridge-example}
By Proposition~\ref{prop:constant-potential-flat}, no extended frame with constant coefficients satisfies the hypothesis of Theorem~\ref{thm:bridge}. The family of Example~\ref{ex:sinh-gordon} satisfies it at every point, since there $A^1B^2-A^2B^1=\tfrac12\omega_{xy}\ne0$ on all of $\{x<y\}$. For each $\theta\ne0$, Theorem~\ref{thm:bridge} yields on all of $\{x<y\}$ a timelike immersion $g^\theta$ of constant Gaussian curvature $1/\sinh^2\theta$, in asymptotic coordinates $(x,y)$ and with Gauss map $\nu$. Since $\hat F$ is explicit, so are $g^\theta$ and all the members $f^\mu$ of the associated family: they are elementary functions of $x$ and $y$. Since $A^2B^1=-e^{-\omega}QR\ne0$, Corollary~\ref{cor:bridge-cmc} applies everywhere as well. The mean curvature of $g^\theta$ is not constant. Otherwise $g^\theta$ would be isoparametric, and so, by Lemma~\ref{lem:parallel-curvatures}, would be its parallel surface $f^{e^{4\theta}}$. But, by Example~\ref{ex:sinh-gordon}, the Gaussian curvature of $f^\mu$ is not constant. The same applies to $g^\theta_-$: it gives explicit timelike immersions of constant Gaussian curvature $-1/\cosh^2\theta$, whose mean curvature is not constant either, as the explicit formula shows.
\end{remark}

The value $t=\theta$ is not the only distinguished one. The curvature of $g^\theta$ satisfies $\sinh^2\theta=1/K^\theta$, so by Theorem~\ref{thm:parallel}(2) its parallel surface at distance $\theta$ has constant mean curvature; and since parallel displacements compose, that surface is $g^{2\theta}$. The next corollary records this at the level of the frame. The surface $g^{2\theta}$ is again a two-point evaluation of $\hat F$, now with $\exp(2\theta e_3)$ inserted, in which the roles of the two null directions are interchanged. Its nondegeneracy condition is not $K\ne0$ but the nonvanishing of both Hopf differentials.

\begin{corollary}
\label{cor:bridge-cmc}
Let $\hat F$ and $\theta$ be as in Theorem~\ref{thm:bridge} and $\mu=e^{4\theta}$. At every point where $A^2\ne0$ and $B^1\ne0$ --- equivalently, where both Hopf differentials $Q^\mu=-A^1A^2(\mu-1)$ and $R^\mu=B^1B^2(\mu^{-1}-1)$ of the immersion $f^\mu$ of Proposition~\ref{prop:sym-bobenko} are nonzero --- the map
\begin{equation}
\label{eq:bridge-cmc-formula}
g^{2\theta}=\hat F|_{\lambda=e^{4\theta}}\;\exp(2\theta e_3)\;\hat F|_{\lambda=1}^{-1}
\end{equation}
is a timelike immersion of constant mean curvature
\[
H^{2\theta}=-\coth(2\theta)=-\frac{\mu+1}{\mu-1},
\]
with the same Gauss map $\nu=\Ad_Fe_3$, and
\[
(g^{2\theta})^{-1}g^{2\theta}_x=A^2(\mu-1)\,\Ad_F\xi_2,\qquad
(g^{2\theta})^{-1}g^{2\theta}_y=B^1(\mu^{-1}-1)\,\Ad_F\xi_1 .
\]
\end{corollary}
\begin{proof}
Put $t=2\theta$ in Lemma~\ref{lem:interpolation}, so that $e^{2t}=e^{4\theta}=\mu$. The coefficients of $\xi_1$ in~\eqref{eq:interp-x} and of $\xi_2$ in~\eqref{eq:interp-y} then vanish, leaving the two displayed identities, with $p'=A^2(\mu-1)$ and $q'=B^1(\mu^{-1}-1)$ in place of the $p,q$ of Proposition~\ref{prop:sym-bobenko} and with the roles of $\xi_1$ and $\xi_2$ interchanged. By~\eqref{eq:interp-EG} the coordinates are again null, and
\[
2\langle g^{2\theta}_x,g^{2\theta}_y\rangle=2p'q'\langle\xi_2,\xi_1\rangle=p'q'=A^2B^1(\mu-1)(\mu^{-1}-1),
\]
which is nonzero precisely under the stated hypothesis, so $g^{2\theta}$ is a nondegenerate timelike immersion; that $N=g^{2\theta}\Ad_Fe_3$ is its unit normal follows as in Proposition~\ref{prop:sym-bobenko}. For the mean curvature, expand $(g^{2\theta})^{-1}g^{2\theta}_{xy}=v'_x+u'v'$ with $u'=p'\Ad_F\xi_2$ and $v'=q'\Ad_F\xi_1$. Using $[F^{-1}F_x,\xi_1]=2A^3\xi_1-A^2e_3$ and the matrix product $\xi_2\xi_1=\tfrac12(\Id-e_3)$ from~\eqref{eq:xi-products},
\[
v'_x=\Ad_F\big((\partial_xq'+2q'A^3)\xi_1-q'A^2e_3\big),\qquad u'v'=\tfrac{p'q'}2\big(\Id-\Ad_Fe_3\big),
\]
so pairing with $N$ leaves $\langle g^{2\theta}_{xy},N\rangle=-q'A^2-\tfrac{p'q'}2$, and therefore
\[
H^{2\theta}=\frac{2\langle g^{2\theta}_{xy},N\rangle}{p'q'}=-\frac{2A^2}{p'}-1=-\frac{2}{\mu-1}-1=-\frac{\mu+1}{\mu-1},
\]
a constant. \end{proof}

\begin{remark}[the three distinguished values, and their nondegeneracy conditions]
\label{rem:three-values}
Fix $\mu=e^{4\theta}$ with $\theta\ne0$ and let $t$ run. By Lemma~\ref{lem:interpolation} the family $g^t=(f^\mu)^t$ is governed by the four coefficients $A^1,A^2,B^1,B^2$, and three values of $t$ are distinguished:
\[
\begin{array}{llll}
t=0: & \text{only }A^1,B^2\text{ survive}, & (x,y)\text{ null}, & \text{CMC with }H=\coth 2\theta;\\
t=\theta: & \text{no factor vanishes}, & (x,y)\text{ asymptotic}, & \text{CGC with }K=1/\sinh^2\theta;\\
t=2\theta: & \text{only }A^2,B^1\text{ survive}, & (x,y)\text{ null}, & \text{CMC with }H=-\coth 2\theta.
\end{array}
\]
The two constant mean curvature members are exchanged by $t\mapsto2\theta-t$, which also reverses the sign of~\eqref{eq:interp-LN}, and the constant Gaussian curvature member sits at the fixed point of that involution. Geometrically, since $(g^\theta)^s=g^{\theta+s}$, the involution is reflection through the constant Gaussian curvature surface along its normal geodesics. The surfaces $f^\mu=(g^\theta)^{-\theta}$ and $g^{2\theta}=(g^\theta)^{\theta}$ are its two parallel surfaces at distance $\theta$, one on either side. They have opposite constant mean curvatures $\pm\coth2\theta$, and all three share the Gauss map $\nu$. This is the classical pair of constant mean curvature surfaces parallel to a surface of constant positive curvature~\cite[Chap.~I, Prop.~3.2 and Cor.~3.4]{Tenenblat98}, which for timelike surfaces is in~\cite[Rem.~3.5]{Zuo06}. Their nondegeneracy conditions are genuinely different: $t=0$ needs $A^1B^2\ne0$, which is regularity of the frame and is assumed throughout; $t=2\theta$ needs $A^2B^1\ne0$, that is, both Hopf differentials nonzero; and $t=\theta$ needs $A^1B^2-A^2B^1\ne0$, that is $K\ne0$, by Proposition~\ref{prop:degenerate-locus}. The last two are not independent, and they fail in opposite directions: under regularity, $A^1B^2=A^2B^1$ forces $A^2B^1=A^1B^2\ne0$, so
\[
K=0\ \Longrightarrow\ A^2B^1\ne0 .
\]
Wherever Theorem~\ref{thm:bridge} fails, therefore, Corollary~\ref{cor:bridge-cmc} applies. In particular the flat example of \S\ref{sec:cmc-example}, which by Proposition~\ref{prop:constant-potential-flat} is excluded from Theorem~\ref{thm:bridge} at every point, satisfies the hypothesis of Corollary~\ref{cor:bridge-cmc} everywhere, since there $A^2B^1=\tfrac14$.
\end{remark}

The members of the family $g^t$ other than the three distinguished ones are neither of constant mean nor of constant Gaussian curvature in general, but they are not arbitrary either.

\begin{proposition}[the interpolating family is linear Weingarten; cf.~{\cite[Thm.~3.4]{Zuo06}}]
\label{prop:lw}
Let $\hat F$ be a regular admissible extended frame, let $\mu\ne0,1$, put $H=(\mu+1)/(\mu-1)$, and for $t\in\mathbb R$ let $g^t$ be as in Lemma~\ref{lem:interpolation}. Put
\[
a_t=\sinh t\,\big(H\sinh t-\cosh t\big),\qquad b_t=H\sinh2t-\cosh2t,\qquad c_t=H\cosh2t-\sinh2t .
\]
At every point where $g^t$ is an immersion, its Gaussian curvature $K^t$ and its mean curvature $H^t$ satisfy
\begin{equation}
\label{eq:lw}
a_t\,K^t+b_t\,H^t+c_t=0 .
\end{equation}
The coefficients are constant on $\Omega$, they are not all zero, and $b_t^2-4a_t(c_t-a_t)=1$. Moreover $(x,y)$ are null coordinates for the quadratic form $b_t\,I^t+2a_t\,II^t$, where $I^t$ and $II^t$ are the first and second fundamental forms of $g^t$; hence, wherever that form is nondegenerate, the Gauss map $\nu$ of $g^t$ is harmonic with respect to its conformal structure. For $\mu=e^{4\theta}$ the relation~\eqref{eq:lw} reads
\[
\sinh t\,\sinh(2\theta-t)\,K^t+\sinh2(\theta-t)\,H^t=\cosh2(\theta-t).
\]
\end{proposition}
\begin{proof}
By Lemma~\ref{lem:bridge-identity}, $g^t$ is the parallel surface at distance $t$ of $f^\mu$, which has constant mean curvature $H$ by Proposition~\ref{prop:sym-bobenko}; let $K$ denote the Gaussian curvature of $f^\mu$. By Lemma~\ref{lem:parallel-curvatures}, at the points where $g^t$ is an immersion $\Delta_t=K\sinh^2t-H\sinh2t+\cosh2t\ne0$ and
\[
\Delta_tK^t=K,\qquad \Delta_tH^t=c_t-\tfrac12K\sinh2t,\qquad \Delta_t=K\sinh^2t-b_t .
\]
Therefore
\[
\Delta_t\big(a_tK^t+b_tH^t+c_t\big)=K\big(a_t-\tfrac12b_t\sinh2t+c_t\sinh^2t\big)+b_tc_t-c_tb_t ,
\]
and the first bracket vanishes: by the addition formulas $\sinh2t\cosh t-\cosh2t\sinh t=\sinh t$ and $\cosh2t\cosh t-\sinh2t\sinh t=\cosh t$,
\[
\tfrac12b_t\sinh2t-c_t\sinh^2t=\sinh t\big(b_t\cosh t-c_t\sinh t\big)=\sinh t\big(H\sinh t-\cosh t\big)=a_t .
\]
This is~\eqref{eq:lw}. Since $c_t^2-b_t^2=H^2-1\ne0$, the coefficients are not all zero. Writing $p=\sinh t$ and $q=\cosh t$, one has $a_t=p(Hp-q)$, $c_t-a_t=q(Hq-p)$ and $b_t=2Hpq-(p^2+q^2)$, so
\[
b_t^2-4a_t(c_t-a_t)=\big(2Hpq-(p^2+q^2)\big)^2-4pq\,(Hp-q)(Hq-p)=(p^2+q^2)^2-4p^2q^2=1 .
\]
For the statement about $(x,y)$, put $m=e^{2t}$. Substituting $H=(\mu+1)/(\mu-1)$ gives
\[
b_t=\frac{m-\mu m^{-1}}{\mu-1},\qquad a_t=\frac{(m-1)(1-\mu m^{-1})}{2(\mu-1)} ,
\]
and then, by~\eqref{eq:interp-EG} and~\eqref{eq:interp-LN},
\[
b_t\langle g^t_x,g^t_x\rangle+2a_t\langle g^t_{xx},N^t\rangle=\frac{A^1A^2}{\mu-1}\,(m-\mu m^{-1})(m-1)\Big[\big(\mu m^{-1}-1\big)+\big(1-\mu m^{-1}\big)\Big]=0,
\]
and in the same way $b_t\langle g^t_y,g^t_y\rangle+2a_t\langle g^t_{yy},N^t\rangle=0$. So $b_tI^t+2a_tII^t$ has no $dx^2$ and no $dy^2$ term. The Gauss map of $g^t$ is $\nu$ by~\eqref{eq:gauss-invariant}, and $[\nu,\nu_{xy}]=0$ by Theorem~\ref{thm:cmc-harmonic}; by~\eqref{eq:harmonic} this is harmonicity for any Lorentz structure in which $(x,y)$ are null.
\end{proof}

\begin{remark}
\label{rem:lw}
Relation~\eqref{eq:lw} uses nothing beyond Lemma~\ref{lem:parallel-curvatures}. In a Riemannian space form, the statement that a surface parallel to a linear Weingarten surface is again linear Weingarten, with the formula for the new coefficients, is~\cite[Chap.~I, Cor.~3.6]{Tenenblat98}; for timelike surfaces in Lorentzian space forms, $\Hyp$ included, the surfaces parallel to one of constant Gaussian curvature are shown to be linear Weingarten, with explicit coefficients, in~\cite[Thm.~3.4]{Zuo06}. Relation~\eqref{eq:lw} is that statement with the family parametrized from its constant mean curvature member, and no novelty is claimed for it. What the present setting adds is that the whole pencil is the family of slices $t=\mathrm{const}$ of the single map $(x,y,t)\mapsto\hat F|_{\lambda=\mu}\exp(te_3)\hat F|_{\lambda=1}^{-1}$, that all its members have the same Gauss map, and that this map is harmonic for the conformal structure of $b_tI^t+2a_tII^t$, which interpolates between the first fundamental form ($a_t=0$) and the second ($b_t=0$). The distinguished members of Remark~\ref{rem:three-values} are exactly these: $a_t=0$ if and only if $t=0$ or $\tanh t=1/H$, and $b_t=0$ if and only if $\tanh2t=1/H$; for $\mu=e^{4\theta}$ these are $t=0$, $t=2\theta$ and $t=\theta$. For $\mu<0$, that is $|H|<1$, neither $\tanh t=1/H$ nor $\tanh2t=1/H$ has a real solution: unless $K$ is constant, $f^\mu$ is then the only member of constant mean curvature and no member has constant Gaussian curvature, in agreement with Theorem~\ref{thm:parallel}.
\end{remark}

The parameters are related by the chain
\begin{equation}
\label{eq:parameter-chain}
\mu=e^{4\theta}
\;\longrightarrow\;
H=\coth(2\theta)
\;\longrightarrow\;
\rho=-\coth\theta
\;\longrightarrow\;
K^\theta=\rho^2-1=\frac1{\sinh^2\theta},
\end{equation}
where $\rho$ is the parameter of Proposition~\ref{prop:cgc-converse}. Every step is an explicit elementary function of the previous one, and each is invertible on its range: from a desired curvature $K^\theta>0$ one recovers $\sinh^2\theta=1/K^\theta$, then $\theta$ up to sign, then the spectral parameter $\mu=e^{4\theta}$ at which the CMC frame must be evaluated. For the negative branch of Theorem~\ref{thm:bridge} the chain ends instead with $\rho=-\tanh\theta$ and $K=-1/\cosh^2\theta$.

\subsection{The geometric Cauchy problem for constant mean curvature surfaces}
\label{sec:cauchy-cmc}

Lemma~\ref{lem:cmc-twisted} also solves the geometric Cauchy problem on the constant mean curvature side, by the same two steps as Corollary~\ref{cor:geometric-cauchy}. For $\tilde\nu\in\Sph$ let $P_\pm$ denote the projections of $\tilde\nu^\perp$ onto the two null lines on which $\tfrac12\operatorname{ad}_{\tilde\nu}$ acts by $\pm1$; these exist by Proposition~\ref{prop:bracket-identities}(v).

\begin{corollary}
\label{cor:cauchy-cmc}
Let $I\subset\mathbb R$ be an open interval, and let $\tilde f:I\to\Hyp$ and $\tilde\nu:I\to\Sph$ satisfy $\langle\tilde f^{-1}\tilde f',\tilde\nu\rangle=0$, with $\tilde f'$ nowhere null. Let $H$ be a constant with $|H|>1$, let $\theta$ be defined by $\tanh(2\theta)=1/H$, and put $\mu=e^{4\theta}$ and
\[
N_1=-\frac2{\mu-1}\,P_+\big(\tilde f^{-1}\tilde f'\big)+P_-\big(\tilde\nu'\big)-\frac2{\mu^{-1}-1}\,P_-\big(\tilde f^{-1}\tilde f'\big).
\]
Let $\hat F$ be the extended frame of Theorem~\ref{thm:cauchy} for the data $N_0=\tilde\nu$ and $N_1$, with frame $F$ along $I$. Then, on a neighbourhood $\Omega$ of the diagonal,
\[
f=L\,\hat F|_{\lambda=e^{2\theta}}\,\exp(-\theta e_3)\,\hat F|_{\lambda=1}^{-1},\qquad L=\tilde f(x_0)F(x_0)\exp(\theta e_3)F(x_0)^{-1},
\]
is a timelike immersion of constant mean curvature $H$, with Gauss map $\nu=\Ad_{\hat F|_{\lambda=1}}e_3$ and with $(x,y)$ as null coordinates for its first fundamental form, and
\[
f(x,x)=\tilde f(x),\qquad \nu(x,x)=\tilde\nu(x)\qquad(x\in I).
\]
\end{corollary}
\begin{proof}
Write $G=\hat F$, and define $A^i,B^i$ by $\beta_0=A^3e_3$, $\beta_1=A^1\xi_1+A^2\xi_2$, $\gamma_0=B^3e_3$ and $\gamma_{-1}=B^1\xi_1+B^2\xi_2$ in its Maurer--Cartan form~\eqref{eq:cgc-extended}. Reading the proof of Lemma~\ref{lem:cmc-twisted} backwards, $\hat F_c(\lambda):=G(\sqrt\lambda)\exp(-\tfrac14\log\lambda\;e_3)$, for $\lambda>0$, has a Maurer--Cartan form of the shape~\eqref{eq:admissible} with these coefficients. The proof of Proposition~\ref{prop:sym-bobenko} uses only that shape and the regularity $A^1B^2\ne0$. So wherever $A^1B^2\ne0$, the map $\hat F_c(\mu)\hat F_c(1)^{-1}=G(e^{2\theta})\exp(-\theta e_3)G(1)^{-1}$ is a timelike immersion with null coordinates $(x,y)$, Gauss map $\nu$ and constant mean curvature $(\mu+1)/(\mu-1)=\coth2\theta=H$. Left translation by $L$ changes none of this.

On the diagonal, $\nu=\tilde\nu$ and $\nu_x=N_1$ by Theorem~\ref{thm:cauchy}, so $\nu_y=\tilde\nu'-N_1$. By~\eqref{eq:nu-derivs} and~\eqref{eq:xi-brackets}, $\nu_x=\Ad_F(-2A^1\xi_1+2A^2\xi_2)$ and $\nu_y=\Ad_F(-2B^1\xi_1+2B^2\xi_2)$, and $\Ad_F\xi_1$, $\Ad_F\xi_2$ span the lines on which $\tfrac12\operatorname{ad}_\nu$ acts by $+1$ and $-1$. Hence $P_+\nu_x=-2A^1\Ad_F\xi_1$ and $P_-\nu_y=2B^2\Ad_F\xi_2$, and by Proposition~\ref{prop:sym-bobenko}
\[
f^{-1}f_x=A^1(\mu-1)\Ad_F\xi_1=-\tfrac12(\mu-1)P_+\nu_x,\qquad f^{-1}f_y=B^2(\mu^{-1}-1)\Ad_F\xi_2=\tfrac12(\mu^{-1}-1)P_-\nu_y .
\]
By the definition of $N_1$, $P_+\nu_x=-\tfrac2{\mu-1}P_+(\tilde f^{-1}\tilde f')$ and $P_-\nu_y=P_-\tilde\nu'-P_-N_1=\tfrac2{\mu^{-1}-1}P_-(\tilde f^{-1}\tilde f')$. Both are nonzero because $\tilde f'$ is not null, so $A^1B^2\ne0$ on the diagonal, and hence on a neighbourhood of it. Moreover, along the diagonal
\[
f^{-1}f_t=f^{-1}f_x+f^{-1}f_y=P_+(\tilde f^{-1}\tilde f')+P_-(\tilde f^{-1}\tilde f')=\tilde f^{-1}\tilde f',
\]
since $\tilde f^{-1}\tilde f'\perp\tilde\nu$. Finally $\hat F|_\lambda(x_0,x_0)=F(x_0)$ for every $\lambda$, so that $f(x_0,x_0)=\tilde f(x_0)$ by the choice of $L$. Hence $f$ and $\tilde f$ agree along the diagonal.
\end{proof}

Unlike Corollary~\ref{cor:geometric-cauchy}, this requires neither that $\nu$ be an immersion nor any passage through parallel surfaces: the constant mean curvature surface may be flat. The case $|H|<1$ is covered by the same formula, evaluated at an imaginary value of the twisted parameter.

\begin{remark}[the case $|H|<1$]
\label{rem:H-small}
Let $G=\hat F$ be the extended frame of Theorem~\ref{thm:cauchy}, with twisted parameter $\zeta$. The potentials there are Laurent polynomials in $\zeta$, so $X$ and $Y$ are holomorphic in $\zeta\in\mathbb C^*$; and $G=XH_-=YH_+^{-1}$ with $H_-$ holomorphic outside the unit disc and $H_+$ inside it, so $G$ is holomorphic in $\zeta\in\mathbb C^*$ as well. This is what is already used when $G$ is evaluated at the real value $e^{2\theta}\ne\pm1$ in Corollary~\ref{cor:cauchy-cmc}. Now let $|H|<1$, so that $\mu=(H+1)/(H-1)<0$, and put
\[
\zeta=i\sqrt{|\mu|},\qquad \theta=\tfrac12\log\zeta=\tfrac14\log|\mu|+\tfrac{i\pi}4 ,
\]
so that $e^{4\theta}=\mu$ and $\tanh2\theta=1/H$ as before. Let $N_1$ be defined by the formula of Corollary~\ref{cor:cauchy-cmc}, which involves only $\mu$. Then
\[
f=L\,G(\zeta)\,\exp(-\theta e_3)\,G(1)^{-1},\qquad L=\tilde f(x_0)F(x_0)\exp(\theta e_3)F(x_0)^{-1},
\]
takes values in $\SLR$, and it is the timelike immersion of constant mean curvature $H$ with $f(x,x)=\tilde f(x)$ and $\nu(x,x)=\tilde\nu(x)$. Indeed, the computation in the proof of Lemma~\ref{lem:cmc-twisted} is valid for complex $\zeta$, so $G(\zeta)\exp(-\theta e_3)$ has the Maurer--Cartan form~\eqref{eq:admissible} at $\lambda=\zeta^2=\mu$, with the real coefficients $A^i,B^i$ of the proof of Corollary~\ref{cor:cauchy-cmc}. Hence $R:=F(x_0)\exp(\theta e_3)F(x_0)^{-1}G(\zeta)\exp(-\theta e_3)$ solves a real linear system with the real initial value $R(x_0,x_0)=F(x_0)$. So $R$ is real: it is the value at $\lambda=\mu$ of the admissible extended frame whose value at $\lambda=1$ is $G(1)$. Since $f=\tilde f(x_0)\,R\,G(1)^{-1}$, Proposition~\ref{prop:sym-bobenko}, which allows $\mu<0$, and the second half of the proof of Corollary~\ref{cor:cauchy-cmc} apply verbatim.

The complex constant inserted is $\exp(-\theta e_3)=\exp(-\tfrac14\log|\mu|\,e_3)\,\operatorname{diag}(e^{i\pi/4},e^{-i\pi/4})$. A formula of the same shape, with a constant diagonal matrix of entries $e^{\mp i\pi/4}$ inserted between the values of one extended frame at $\lambda$ and at $i\lambda$, gives the spacelike maximal surfaces in $\Hyp$~\cite[(2.5)]{KZ26}; see also~\cite{Ogata17}. So constant mean curvature surfaces with $|H|>1$ and with $|H|<1$ are evaluations of the same twisted frame, at real and at imaginary values of $\zeta$ respectively, and only the former have constant Gaussian curvature surfaces among their parallel surfaces (Remark~\ref{rem:lw}).
\end{remark}

\begin{remark}[the flat limit]
\label{rem:flat-limit}
The constructions above degenerate to the flat ones recalled in the introduction. For $\mu>1$ put $c=2/(\mu-1)$ and $\Psi_c(X)=c(X-\Id)$. Since $\det X=1$ on $\Hyp$ and $\langle X,\Id\rangle=-\tfrac12\operatorname{tr}X$, the map $\Psi_c$ carries $\Hyp$ homothetically, with ratio $c$, onto the quadric $\langle Y+c\Id,Y+c\Id\rangle=-c^2$. This is a Lorentzian space form of curvature $-1/c^2=-(\mu-1)^2/4$, and it converges to the hyperplane $\sltwo\cong\mathbb R^{2,1}$ as $\mu\to1$. A homothety of ratio $c$ divides mean curvatures by $c$ and Gaussian curvatures by $c^2$. So, with $\mu=e^{4\theta}$, the images of $f^\mu$, $g^\theta$ and $g^{2\theta}$ have constant curvatures
\[
H=\tfrac{\mu+1}2,\qquad K=\frac{(\mu-1)^2}{4\sinh^2\theta}=4\mu\cosh^2\theta,\qquad H=-\tfrac{\mu+1}2 ,
\]
using $e^{4\theta}-1=2e^{2\theta}\sinh2\theta$. These tend to $1$, $4$ and $-1$. This is the configuration of~\cite[\S4]{Ino98}, where the parallel surface at distance $1/(2H)$ of a timelike surface of constant mean curvature $H$ in $\mathbb R^{2,1}$ has constant Gaussian curvature $4H^2$ (see also~\cite[\S2]{Kokubu17}), and it is the Euclidean case $\bar K=0$ of~\cite[Chap.~I, Cor.~3.4 and Example~3.5]{Tenenblat98}. Moreover, $\hat F$ being real analytic in $\lambda$,
\[
\Psi_c\circ f^\mu=\frac2{\mu-1}\Big(\hat F|_\mu\hat F|_1^{-1}-\hat F|_1\hat F|_1^{-1}\Big)\longrightarrow 2\,\frac{\partial\hat F}{\partial\lambda}\Big|_{\lambda=1}\hat F|_{\lambda=1}^{-1}\in\sltwo ,
\]
which is the first term of the Sym formula by which timelike constant mean curvature surfaces in $\mathbb R^{2,1}$ are recovered from their extended frames~\cite[\S2]{DIT02},~\cite[Lem.~2.3]{BS13}. The differentiation in $\lambda$ of the flat theory is thus what remains of the two-point evaluation when the ambient group is flattened to its Lie algebra and the two evaluation points merge; the same phenomenon relates the constructions in $S^3$ and in $\mathbb E^3$~\cite[\S4.1]{BIK14}. Finally, the flat relation $K=4H^2$ is deformed in $\Hyp$: with $H=\coth2\theta$ and $K=1/\sinh^2\theta$, and $\theta>0$,
\begin{equation}
\label{eq:K-vs-H}
K=4\cosh^2\!\theta\,\big(H^2-1\big)=2\big(H^2-1\big)+2H\sqrt{H^2-1}\,,
\end{equation}
which is asymptotic to $4H^2$ as $H\to\infty$, that is, as $\mu\to1$.
\end{remark}

\section{Discussion and open questions}
\label{sec:conclusion}

Lemma~\ref{lem:cmc-twisted} shows that the two constructions of this paper differ, at the level of frames, only by a normalization, and all the surfaces obtained are two-point evaluations of one twisted frame. We close with two questions this leaves open.

\begin{problem}[the range $K<-1$]
\label{q:branch}
Timelike surfaces of constant Gaussian curvature $K<-1$ are not treated in this paper. For them $\det II=(K+1)\det I>0$, so the second fundamental form is definite and the conformal structure it defines on the domain is Riemannian, not Lorentzian. At the level of the Gauss--Codazzi equations this range is included in~\cite{Xia07} and, for real distinct principal curvatures, in~\cite[Chap.~I, \S2]{Tenenblat98}. In the flat ambient space the timelike surfaces whose Gauss map is harmonic into $\Sph$ for a Riemannian structure on the domain belong to a different real form of the same loop algebra, with the loop parameter on the unit circle~\cite[case~$(C,2)$]{Kob11}, and the Birkhoff decomposition is replaced by an Iwasawa decomposition, which for a noncompact real form is again not global~\cite{BRS10}. Is the normal Gauss map of a timelike surface in $\Hyp$ with $K<-1$ harmonic for the conformal structure of the second fundamental form exactly when $K$ is constant, and are those surfaces two-point evaluations of an extended frame, as in Proposition~\ref{prop:cgc-sym}? We have not investigated this.
\end{problem}

\begin{problem}[the degenerate locus and the singularities]
\label{q:degenerate}
Proposition~\ref{prop:degenerate-locus} identifies the degenerate locus of Theorem~\ref{thm:bridge} with the locus where the Gauss map fails to be an immersion. It does not say how large that locus can be: by~\eqref{eq:K-cmc} the question is how the solutions of the Gauss equation distribute the zeros of $\omega_{xy}$, and we have not investigated it. The behaviour of the constructions there, and at the points where the Birkhoff decomposition used in Theorem~\ref{thm:cauchy} leaves the big cell, is the setting for the singularities of these surfaces, which are deliberately left for future work; for the flat ambient space this is the subject of~\cite{BS14}.
\end{problem}

\subsection*{Acknowledgements}
Part of this work originates in the author's doctoral research, supported by the Independent Research Fund Denmark (DFF), grant number 9040-00196B. The author used Claude (Anthropic), an AI assistant, for literature search, for the symbolic and numerical checking of computations, and for editing the text; all mathematical content was reviewed by the author, who takes full responsibility for it.

\end{document}